\documentclass[12pt,a4paper,english,american]{article}
\usepackage[T1]{fontenc}
\usepackage[latin9]{inputenc}
\usepackage{amsmath}
\usepackage{graphicx}
\usepackage{amssymb}
\usepackage{esint}

\makeatletter
\newcommand{\lyxaddress}[1]{
\par {\raggedright #1
\vspace{1.4em}
\noindent\par}
}

\usepackage{amsthm}
\usepackage{mathrsfs}

\newtheorem{theorem}{Theorem}
\newtheorem{proposition}[theorem]{Proposition}
\newtheorem{lemma}[theorem]{Lemma}
\newtheorem{corollary}[theorem]{Corollary}
\theoremstyle{remark}
\newtheorem{remark}[theorem]{Remark}
\newtheorem{example}[theorem]{Example}
\newtheorem*{question*}{QUESTION}
\newtheorem{definition}[theorem]{Definition}

\newcommand{\diag}{\mathop\mathrm{diag}\nolimits}

\newcommand{\Dom}{\mathop\mathrm{Dom}\nolimits}
\newcommand{\der}{\mathop\mathrm{der}\nolimits}
\newcommand{\Ran}{\mathop\mathrm{Ran}\nolimits}

\newcommand{\spec}{\mathop\mathrm{spec}\nolimits}

\renewcommand{\Re}{\mathop\mathrm{Re}\nolimits}

\newcommand{\dist}{\mathop\mathrm{dist}\nolimits}

\makeatother

\makeatother

\usepackage{babel}

\begin{document}

\title{The characteristic function for Jacobi matrices with applications}

\author{F.~\v{S}tampach, P.~\v{S}\v{t}ov\'\i\v{c}ek}

\date{{}}

\maketitle

\lyxaddress{Department of Mathematics, Faculty of Nuclear Science, Czech Technical
University in Prague, Trojanova13, 12000 Praha, Czech Republic}
\begin{abstract}
\noindent We introduce a class of Jacobi operators with discrete spectra
which is characterized by a simple convergence condition. With any
operator $J$ from this class we associate a characteristic function
as an analytic function on a suitable domain, and show that its zero
set actually coincides with the set of eigenvalues of $J$ in that
domain. Further we derive sufficient conditions under which the spectrum
of $J$ is approximated by spectra of truncated finite-dimensional
Jacobi matrices. As an application we construct several examples of
Jacobi matrices for which the characteristic function can be expressed
in terms of special functions. In more detail we study the example
where the diagonal sequence of $J$ is linear while the neighboring
parallels to the diagonal are constant.
\end{abstract}
\vskip\baselineskip\emph{Keywords}: infinite Jacobi matrix, spectral
problem, characteristic function

\section{Introduction}

In this paper we introduce and study a class of infinite symmetric
but in general complex Jacobi matrices $\mathcal{J}$ characterized
by a simple convergence condition. This class is also distinguished
by the discrete character of spectra of the corresponding Jacobi operators.
Doing so we extend and generalize an approach to Jacobi matrices which
was originally initiated, under much more restricted circumstances,
in \cite{StampachStovicek}. We refer to \cite{JanasNaboko} for a
rather general analysis of how the character of spectrum of a Jacobi
operator may depend on the asymptotic behavior of weights.

For a given Jacobi matrix $\mathcal{J}$, one constructs a characteristic
function $F_{\mathcal{J}}(z)$ as an analytic function on the domain
$\mathbb{C}_{0}^{\lambda}$ obtained by excluding from the complex
plane the closure of the range of the diagonal sequence $\lambda$
of $\mathcal{J}$ . Under some comparatively simple additional assumptions,
like requiring the real part of $\lambda$ to be semibounded or $\mathcal{J}$
to be real, one can show that $\mathcal{J}$ determines a unique closed
Jacobi operator $J$ on $\ell^{2}(\mathbb{N})$. Moreover, the spectrum
of $J$ in the domain $\mathbb{C}_{0}^{\lambda}$ is discrete and
coincides with the zero set of $F_{\mathcal{J}}(z)$. When establishing
this relationship one may also treat the poles of $F_{\mathcal{J}}(z)$
which occur at the points from the range of the sequence $\lambda$
not belonging to the set of accumulation points, however. In addition,
as an important step of the proof, one makes use of an explicit formula
for the Green function associated with $J$.

Apart of the localization of the spectrum we address too the question
of approximation of the spectrum by spectra of truncated finite-dimensional
Jacobi matrices. For bounded Hermitian Jacobi operators the problem
has been studied, for example in \cite{HartmanWinter,Arveson,IfantisPanagopoulos}.
We are aware of just a few papers, however, bringing some results
in this respect also about unbounded Jacobi operators \cite{IfantisSiafarikas,Ifantis_etal,Malejki}.
Our approach based on employing the characteristic function makes
it possible to derive sufficient conditions under which such an approximation
can be verified. This result partially reproduces and overlaps with
some theorems from \cite{Ifantis_etal}.

The characteristic function as well as numerous formulas throughout
the paper are expressed in terms of a function, called $\mathfrak{F}$,
defined on a subset of the space of complex sequences. In the introductory
part we recall from \cite{StampachStovicek} the definition of $\mathfrak{F}$
and its basic properties which are then completed by various additional
facts. On the other hand, we conclude the paper with some applications
of the derived results. We present several examples of Jacobi matrices
for which the characteristic function can be expressed in terms of
special functions (the Bessel functions or the basic hypergeometric
series). A particular attention is paid to the example where the diagonal
sequence $\lambda$ is linear while the neighboring parallels to the
diagonal are constant. In this case the characteristic equation in
the variable $z$ reads $J_{-z}(2w)=0$, with $w$ being a parameter,
and our main concern is how the spectrum of the Jacobi operator depends
on $w$.

\section{The function $\mathfrak{F}$}

\subsection{Definition and basic properties}

Let us recall from \cite{StampachStovicek} some basic definitions
and properties concerning a function $\mathfrak{F}$ defined on a
subset of the linear space formed by all complex sequences $x=\{x_{k}\}_{k=1}^{\infty}$.
Moreover, we complete this brief overview by a few additional facts.

\begin{definition} Define $\mathfrak{F}:D\rightarrow\mathbb{C}$,
\begin{equation}
\mathfrak{F}(x)=1+\sum_{m=1}^{\infty}(-1)^{m}\sum_{k_{1}=1}^{\infty}\,\sum_{k_{2}=k_{1}+2}^{\infty}\,\dots\,\sum_{k_{m}=k_{m-1}+2}^{\infty}\, x_{k_{1}}x_{k_{1}+1}x_{k_{2}}x_{k_{2}+1}\dots x_{k_{m}}x_{k_{m}+1},\label{eq:defn_F}\end{equation}
where \[
D=\left\{ \{x_{k}\}_{k=1}^{\infty}\subset\mathbb{C};\,\sum_{k=1}^{\infty}|x_{k}x_{k+1}|<\infty\right\} .\]
For a finite number of complex variables we identify $\mathfrak{F}(x_{1},x_{2},\dots,x_{n})$
with $\mathfrak{F}(x)$ where $x=(x_{1},x_{2},\dots,x_{n},0,0,0,\dots)$.
By convention, we also put $\mathfrak{F}(\emptyset)=1$ where $\emptyset$
is the empty sequence. \end{definition}

Let us remark that the value of $\mathfrak{F}$ on a finite complex
sequence can be expressed as the determinant of a finite Jacobi matrix.
Using some basic linear algebra it is easy to show that, for $n\in\mathbb{N}$
and $\{x_{j}\}_{j=1}^{n}\subset\mathbb{C}$, one has \begin{equation}
\mathfrak{F\!}\left(x_{1},x_{2},\dots,x_{n}\right)=\det X_{n}\label{eq:F_simple_det}\end{equation}
where\[
X_{n}=\begin{pmatrix}1 & x_{1}\\
x_{2} & 1 & x_{2}\\
 & \ddots & \ddots & \ddots\\
 &  & \ddots & \ddots & \ddots\\
 &  &  & x_{n-1} & 1 & x_{n-1}\\
 &  &  &  & x_{n} & 1\end{pmatrix}\!.\]

Note that the domain $D$ is not a linear space. One has, however,
$\ell^{2}(\mathbb{N})\subset D$. In fact, the absolute value of the
$m$th summand on the RHS of (\ref{eq:defn_F}) is majorized by the
expression \[
\sum_{\substack{k\in\mathbb{N}^{m}\\
k_{1}<k_{2}<\dots<k_{m}}
}|x_{k_{1}}x_{k_{1}+1}x_{k_{2}}x_{k_{2}+1}\dots x_{k_{m}}x_{k_{m}+1}|\leq\frac{1}{m!}\left(\sum_{j=1}^{\infty}|x_{j}x_{j+1}|\right)^{\! m}.\]
Hence for $x\in D$ one has the estimate \begin{equation}
\left|\mathfrak{F}(x)\right|\leq\exp\!\left(\sum_{k=1}^{\infty}|x_{k}x_{k+1}|\right)\!.\label{eq:F_ineq_exp}\end{equation}

Furthermore, $\mathfrak{F}$ satisfies the relation \begin{equation}
\mathfrak{F}(x)=\mathfrak{F}(x_{1},\dots,x_{k})\,\mathfrak{F}(T^{k}x)-\mathfrak{F}(x_{1},\dots,x_{k-1})x_{k}x_{k+1}\mathfrak{F}(T^{k+1}x),\quad k=1,2,\dots,\label{eq:F_T_recur_k}\end{equation}
where $x\in D$ and $T$ denotes the truncation operator from the
left defined on the space of all sequences, $T(\{x_{n}\}_{n=1}^{\infty})=\{x_{n+1}\}_{n=1}^{\infty}$.
In particular, for $k=1$ one gets the rule \begin{equation}
\mathfrak{F}(x)=\mathfrak{F}(Tx)-x_{1}x_{2}\mathfrak{F}(T^{2}x).\label{eq:F_T_recur}\end{equation}
In addition, one has the symmetry property \[
\mathfrak{F}(x_{1},x_{2},\dots,x_{k-1},x_{k})=\mathfrak{F}(x_{k},x_{k-1},\dots,x_{2},x_{1}).\]
If combined with (\ref{eq:F_T_recur_k}), one gets\begin{equation}
\mathfrak{F}(x_{1},x_{2},\dots,x_{k+1})=\mathfrak{F}(x_{1},x_{2},\dots,x_{k})-x_{k}x_{k+1}\,\mathfrak{F}(x_{1},x_{2},\dots,x_{k-1}).\label{eq:calF_reccur_reverse}\end{equation}

\begin{lemma} \label{lemm:F_fin_lim} For $x\in D$ one has \begin{equation}
\lim_{n\rightarrow\infty}\mathfrak{F}(T^{n}x)=1\label{eq:lim_F_T_n}\end{equation}
and \begin{equation}
\lim_{n\rightarrow\infty}\mathfrak{F}(x_{1},x_{2},\dots,x_{n})=\mathfrak{F}(x).\label{eq:F_fin_lim}\end{equation}
\end{lemma}

\begin{proof} First, similarly as in (\ref{eq:F_ineq_exp}), one
gets the estimate\[
\left|\mathfrak{F}(T^{n}x)-1\right|\leq\exp\!\left(\sum_{k=n+1}^{\infty}|x_{k}x_{k+1}|\right)-1.\]
This shows (\ref{eq:lim_F_T_n}).

Second, in view of (\ref{eq:F_T_recur_k}), the difference $|\mathfrak{F}(x)-\mathfrak{F}(x_{1},x_{2},\dots,x_{n})|$
can be majorized by the expression \[
|1-\mathfrak{F}(T^{n}x)|\exp\!\left(\sum_{k=1}^{\infty}|x_{k}x_{k+1}|\right)+|x_{n}x_{n+1}|\exp\!\left(2\sum_{k=1}^{\infty}|x_{k}x_{k+1}|\right)\!.\]
From here one derives the (rather rough) estimate\begin{equation}
|\mathfrak{F}(x)-\mathfrak{F}(x_{1},x_{2},\dots,x_{n})|\leq2\exp\!\left(2\sum_{k=1}^{\infty}|x_{k}x_{k+1}|\right)\sum_{k=n}^{\infty}|x_{k}x_{k+1}|.\label{eq:estim_F-Ffin}\end{equation}
This shows (\ref{eq:F_fin_lim}). \end{proof}

\begin{proposition} \label{prop:contin_F} The function $\mathfrak{F}$
is continuous on $\ell^{2}(\mathbb{N})$. \end{proposition}

\begin{proof} If $x\in\ell^{2}(\mathbb{N})\subset D$ then from (\ref{eq:estim_F-Ffin})
one derives that, for any $n\in\mathbb{N}$,\[
|\mathfrak{F}(x)-\mathfrak{F}(x_{1},x_{2},\dots,x_{n})|\leq2\,\exp\!\left(2\,\|x\|^{2}\right)\,\|(I-P_{n-1})x\|^{2}\]
where $P_{m}$ stands for the orthogonal projection on $\ell^{2}(\mathbb{N})$
onto the subspace spanned by the first $m$ vectors of the standard
basis. From this estimate and from the fact that $\mathfrak{F}(x_{1},x_{2},\dots,x_{n})$
is a polynomial function the proposition readily follows. \end{proof}

\subsection{Jacobi matrices}

\noindent Let us denote by $\mathcal{J}$ an infinite Jacobi matrix
of the form\[
\mathcal{J}=\begin{pmatrix}\lambda_{1} & w_{1}\\
v_{1} & \lambda_{2} & w_{2}\\
 & v_{2} & \lambda_{3} & w_{3}\\
 &  & \ddots & \ddots & \ddots\end{pmatrix}\]
where $w=\{w_{n}\}_{n=1}^{\infty},v=\{v_{n}\}_{n=1}^{\infty}\subset\mathbb{C}\setminus\{0\}$
and $\lambda=\{\lambda_{n}\}_{n=1}^{\infty}\subset\mathbb{C}$. Provided
any of the sequences is unbounded it is reasonable to distinguish
in the notation between $\mathcal{J}$ and an operator represented
by this matrix. Such an operator $J$ need not be unique, as discussed
in Subsection~\ref{subsec:operatorJ}. Further, by $J_{n}$ we denote
the $n$th truncation of $\mathcal{J}$, i.e. \begin{equation}
J_{n}=\begin{pmatrix}\lambda_{1} & w_{1}\\
v_{1} & \lambda_{2} & w_{2}\\
 & \ddots & \ddots & \ddots\\
 &  & v_{n-2} & \lambda_{n-1} & w_{n-1}\\
 &  &  & v_{n-1} & \lambda_{n}\end{pmatrix}\!.\label{eq:def Jn}\end{equation}

As is well known and in fact quite obvious, any solution $\{x_{k}\}$
of the formal eigenvalue equation\begin{equation}
\lambda_{1}x_{1}+w_{1}x_{2}=zx_{1},\ \ v_{k-1}x_{k-1}+\lambda_{k}x_{k}+w_{k}x_{k+1}=zx_{k}\ \ \text{\text{for }}k\geq2,\label{eq:eigenvalue}\end{equation}
with $z\in\mathbb{C}$, is unambiguously determined by its first component
$x_{1}$. Consequently, any operator $J$ whose matrix equals $\mathcal{J}$
may have only simple eigenvalues.

We wish to show that the characteristic function of a finite Jacobi
matrix $J_{n}$ can be expressed in terms of $\mathfrak{F}$. To this
end, let us introduce the sequences $\{\gamma_{k}^{\pm}\}_{k=1}^{n}$
defined recursively by\begin{equation}
\gamma_{1}^{\pm}=1,\ \gamma_{k+1}^{+}=w_{k}/\gamma_{k}^{-}\,\ \text{and}\ \,\gamma_{k+1}^{-}=v_{k}/\gamma_{k}^{+},\ k\geq1.\label{eq:gamma_pm_seq}\end{equation}
More explicitly, the sequence $\{\gamma_{k}^{-}\}_{k=1}^{n}$ can
be expressed as\[
\gamma_{2k-1}^{-}=\prod_{j=1}^{k-1}\frac{v_{2j}}{w_{2j-1}}\,,\mbox{ }\gamma_{2k}^{-}=v_{1}\prod_{j=1}^{k-1}\frac{v_{2j+1}}{w_{2j}}\,,\mbox{ }k=1,2,3,\ldots.\]
As for the sequence $\{\gamma_{k}^{+}\}_{k=1}^{n}$, the corresponding
expressions are of the same form but with $w$ being replaced by $v$
and vice versa. Note that if $v_{k}=w_{k}$ for all $k=1,2,\dots,n-1$,
then $\gamma_{k}^{-}=\gamma_{k}^{+}$ for all $k=1,2,\dots,n$.

\begin{proposition} \label{prop:charfce_finJ} Let $\{\gamma_{k}^{\pm}\}_{k=1}^{n}$
be the sequences defined in (\ref{eq:gamma_pm_seq}). Then the equality
\begin{equation}
\det(J_{n}-zI_{n})=\left(\prod_{k=1}^{n}(\lambda_{k}-z)\right)\mathfrak{F}\!\left(\frac{\gamma_{1}^{-}\gamma_{1}^{+}}{\lambda_{1}-z},\frac{\gamma_{2}^{-}\gamma_{2}^{+}}{\lambda_{2}-z},\dots,\frac{\gamma_{n}^{-}\gamma_{n}^{+}}{\lambda_{n}-z}\right)\label{eq:char_pol_general}\end{equation}
holds for all $z\in\mathbb{C}$ (after obvious cancellations, the
RHS is well defined even for $z=\lambda_{k}$). \end{proposition}

\begin{proof} Put $\tilde{\lambda}_{k}=\lambda_{k}/\gamma_{k}^{-}\gamma_{k}^{+}$.
As remarked in \cite[Remark~24]{StampachStovicek}, the Jacobi matrix
$J_{n}$ can be decomposed into the product $J_{n}=G_{n}^{-}\tilde{J_{n}}G_{n}^{+}$
where $G_{n}^{\pm}=\diag(\gamma_{1}^{\pm},\gamma_{2}^{\pm},\ldots,\gamma_{n}^{\pm})$
are diagonal matrices, and $\tilde{J_{n}}$ is a Jacobi matrix whose
diagonal equals the sequence $(\tilde{\lambda}_{1},\tilde{\lambda}_{2},\ldots,\tilde{\lambda}_{n})$
and which has all units on the neighboring parallels to the diagonal.
The proposition now readily follows from this decomposition combined
with (\ref{eq:F_simple_det}). \end{proof}

Moreover, with the aid of (\ref{eq:char_pol_general}) and using some
basic calculus from linear algebra one can derive the following formula
for the resolvent.

\begin{proposition} \label{prop:Green_finJ} The matrix entries of
the resolvent $R_{n}(z)=(J_{n}-zI_{n})^{-1}$, with $z\in\mathbb{C}\setminus\spec(J_{n})$,
may be expressed as ($1\leq i,j\leq n$)

\begin{eqnarray}
 &  & \hskip-3emR_{n}(z)_{i,j}=-\,\Omega(i,j)\left(\prod_{l=\min(i,j)}^{\max(i,j)}\,(z-\lambda_{l})\right)^{\!-1}\label{eq:green_fin}\\
 &  & \hskip-1em\qquad\times\,\mathfrak{F}\!\left(\left\{ \frac{\gamma_{l}^{-}\gamma_{l}^{+}}{\lambda_{l}-z}\right\} _{l=1}^{\min(i,j)-1}\right)\mathfrak{F}\!\left(\left\{ \frac{\gamma_{l}^{-}\gamma_{l}^{+}}{\lambda_{l}-z}\right\} _{l=\max(i,j)+1}^{n}\right)\mathfrak{F}\!\left(\left\{ \frac{\gamma_{l}^{-}\gamma_{l}^{+}}{\lambda_{l}-z}\right\} _{l=1}^{n}\right)^{\!-1}\nonumber \end{eqnarray}
where\[
\Omega(i,j)=\begin{cases}
\prod_{l=i}^{j-1}w_{l}, & \mbox{if }i<j,\\
\noalign{\smallskip}1, & \mbox{if }i=j,\\
\noalign{\smallskip}\prod_{l=j}^{i-1}v_{l}, & \mbox{if }i>j.\end{cases}\]
\end{proposition}

In the remainder of the paper we concentrate, however, on symmetric
Jacobi matrices with $v=w$, i.e. we put \[
\mathcal{J}=\begin{pmatrix}\lambda_{1} & w_{1}\\
w_{1} & \lambda_{2} & w_{2}\\
 & w_{2} & \lambda_{3} & w_{3}\\
 &  & \ddots & \ddots & \ddots\end{pmatrix}\!,\]
where $\lambda=\{\lambda_{n}\}_{n=1}^{\infty}\subset\mathbb{C}$ and
$w=\{w_{n}\}_{n=1}^{\infty}\subset\mathbb{C}\setminus\{0\}$. In that
case some definitions introduced above simplify. First of all, one
has $\gamma_{k}^{-}=\gamma_{k}^{+}=\gamma_{k}$ where\[
\gamma_{2k-1}=\prod_{j=1}^{k-1}\frac{w_{2j}}{w_{2j-1}}\,,\mbox{ }\gamma_{2k}=w_{1}\prod_{j=1}^{k-1}\frac{w_{2j+1}}{w_{2j}}\,,\mbox{ }k=1,2,3,\ldots.\]
Then $\gamma_{k}\gamma_{k+1}=w_{k}$.

\subsection{More on the function $\mathfrak{F}$}

In \cite{StampachStovicek} one can find two examples of special functions
expressed in terms of $\mathfrak{F}$. The first example is concerned
with the Bessel functions of the first kind. In more detail, for $w,\nu\in\mathbb{C}$,
$\nu\notin-\mathbb{N}$, one has \begin{equation}
J_{\nu}(2w)=\frac{w^{\nu}}{\Gamma(\nu+1)}\,\mathfrak{F}\!\left(\left\{ \frac{w}{\nu+k}\right\} _{k=1}^{\infty}\right)\!.\label{eq:BesselJ_rel_F}\end{equation}
Notice that jointly with (\ref{eq:F_ineq_exp}) this implies\begin{equation}
\left|J_{\nu}(2w)\right|\leq\left|\frac{w^{\nu}}{\Gamma(\nu+1)}\right|\,\exp\!\left(\sum_{k=1}^{\infty}\left|\frac{w^{2}}{(\nu+k)(\nu+k+1)}\right|\right)\!.\label{eq:Bessel_estim}\end{equation}

In the second example one shows that the formula \begin{equation}
\mathfrak{F}\!\left(\left\{ t^{k-1}w\right\} _{k=1}^{\infty}\right)=1+\sum_{m=1}^{\infty}(-1)^{m}\,\frac{t^{m(2m-1)}w^{2m}}{(1-t^{2})(1-t^{4})\dots(1-t^{2m})}={}_{0}\phi_{1}(;0;t^{2},-tw^{2})\label{eq:F_geom_tw}\end{equation}
holds for $t,w\in\mathbb{C}$, $|t|<1$. Here $_{0}\phi_{1}$ is the
basic hypergeometric series (also called q-hypergeometric series)
being defined by\[
_{0}\phi_{1}(;b;q,z)=\sum_{k=0}^{\infty}\frac{q^{k(k-1)}}{(q;q)_{k}(b;q)_{k}}\, z^{k},\]
and\[
(a;q)_{k}=\prod_{j=0}^{k-1}\left(1-aq^{j}\right),\mbox{ }k=0,1,2,\ldots,\]
is the $q$-Pochhammer symbol; see \cite{GasperRahman} for more details.

In this connection let us recall one more identity proved in \cite[Lemma~9]{StampachStovicek},
namely\begin{eqnarray*}
 &  & u_{1}\,\mathfrak{F}\left(u_{2},u_{3},\ldots,u_{n}\right)\mathfrak{F}\left(v_{1},v_{2},\ldots,v_{n}\right)-v_{1}\,\mathfrak{F}\left(u_{1},u_{2},\ldots,u_{n}\right)\mathfrak{F}\left(v_{2},v_{3},\ldots,v_{n}\right)\\
 &  & =\,\sum_{j=1}^{n}\left(\prod_{k=1}^{j-1}u_{k}v_{k}\right)\left(u_{j}-v_{j}\right)\mathfrak{F}\left(u_{j+1},u_{j+2},\ldots,u_{n}\right)\mathfrak{F}\left(v_{j+1},v_{j+2},\ldots,v_{n}\right).\end{eqnarray*}
For the particular choice\[
u_{k}=\frac{w}{\mu+k}\,,\ v_{k}=\frac{w}{\nu+k}\,,\text{ }1\leq k\leq n,\]
one can consider the limit $n\to\infty$. Using (\ref{eq:BesselJ_rel_F})
and (\ref{eq:Bessel_estim}) one arrives at the equation\begin{equation}
J_{\mu}(2w)J_{\nu+1}(2w)-J_{\mu+1}(2w)J_{\nu}(2w)=\frac{\mu-\nu}{w}\,\sum_{j=1}^{\infty}J_{\mu+j}(2w)J_{\nu+j}(2w).\label{eq:sum_Bessel_mn}\end{equation}

Definition (\ref{eq:defn_F}) can naturally be extended to more general
ranges of indices. For any sequence $\left\{ x_{n}\right\} _{n=N_{1}}^{N_{2}}$,
$N_{1},N_{2}\in\mathbb{Z}\cup\{-\infty,+\infty\}$, $N_{1}\leq N_{2}+1$,
(if $N_{1}=N_{2}+1\in\mathbb{Z}$ then the sequence is considered
as empty) such that\[
\sum_{k=N_{1}}^{N_{2}-1}\left|x_{k}x_{k+1}\right|<\infty\]
one can define\[
\mathfrak{F}\!\left(\left\{ x_{k}\right\} _{k=N_{1}}^{N_{2}}\right)=1+\sum_{m=1}^{\infty}(-1)^{m}\sum_{k\in\mathcal{I}(N_{1},N_{2},m)}x_{k_{1}}x_{k_{1}+1}x_{k_{2}}x_{k_{2}+1}\ldots x_{k_{m}}x_{k_{m}+1}\]
where\[
\mathcal{I}(N_{1},N_{2},m)=\left\{ k\in\mathbb{Z}^{m};\, k_{j}+2\leq k_{j+1}\text{ }\text{for}\ 1\leq j\leq m-1,\ N_{1}\leq k_{1},\ k_{m}<N_{2}\right\} .\]

With this definition one can generalize the rule (\ref{eq:F_T_recur_k}).
Now one has\begin{equation}
\mathfrak{F}\!\left(\left\{ x_{k}\right\} _{k=N_{1}}^{N_{2}}\right)=\mathfrak{F}\!\left(\left\{ x_{k}\right\} _{k=N_{1}}^{n}\right)\mathfrak{F}\!\left(\left\{ x_{k}\right\} _{k=n+1}^{N_{2}}\right)-x_{n}x_{n+1}\,\mathfrak{F}\!\left(\left\{ x_{k}\right\} _{k=N_{1}}^{n-1}\right)\mathfrak{F}\!\left(\left\{ x_{k}\right\} _{k=n+2}^{N_{2}}\right)\label{eq:F_recur_general}\end{equation}
provided $n\in\mathbb{Z}$ satisfies $N_{1}\leq n<N_{2}$.

This extension also opens the way for applications of the function
$\mathfrak{F}$ to bilateral difference equations. Suppose that sequences
$\left\{ w_{n}\right\} _{n=-\infty}^{\infty}$ and $\left\{ \zeta_{n}\right\} _{n=-\infty}^{\infty}$
are such that $w_{n}\neq0$, $\zeta_{n}\neq0$ for all $n$ and

\[
\sum_{k=-\infty}^{\infty}\left|\frac{w_{k}^{\,2}}{\zeta_{k}\zeta_{k+1}}\right|<\infty.\]
Consider the difference equation\begin{equation}
w_{n}u_{n+1}-\zeta_{n}u_{n}+w_{n-1}u_{n-1}=0,\ n\in\mathbb{Z}.\label{eq:diff_eq_inf}\end{equation}
Define the sequence $\left\{ \mathcal{P}_{n}\right\} _{n\in\mathbb{Z}}$
by $\mathcal{P}_{0}=1$ and $\mathcal{P}_{n+1}=(w_{n}/\zeta_{n+1})\mathcal{P}_{n}$
for all $n$. Hence

\[
\mathcal{P}_{n}=\prod_{k=1}^{n}\frac{w_{k-1}}{\zeta_{k}}\text{ }\ \text{for}\ n>0,\ \mathcal{P}_{0}=1,\ \mathcal{P}_{n}=\prod_{k=n+1}^{0}\frac{\zeta_{k}}{w_{k-1}}\text{ }\ \text{for}\ n<0.\]
The sequence $\{\gamma_{n}\}_{n\in\mathbb{Z}}$ is again defined so
that $\gamma_{1}=1$ and $\gamma_{n}\gamma_{n+1}=w_{n}$ for all $n\in\mathbb{Z}$.
Hence\[
\gamma_{2k-1}=\prod_{j=1}^{k-1}\frac{w_{2j}}{w_{2j-1}}\,,\mbox{ }\gamma_{2k}=w_{1}\prod_{j=1}^{k-1}\frac{w_{2j+1}}{w_{2j}}\,,\mbox{ }\text{for}\ k=1,2,3,\ldots,\]
and\[
\gamma_{2k-1}=\prod_{j=k}^{0}\frac{w_{2j-1}}{w_{2j}}\,,\mbox{ }\gamma_{2k}=w_{1}\prod_{j=k}^{0}\frac{w_{2j}}{w_{2j+1}}\,,\mbox{ }\text{for}\ k=0,-1,-2,\ldots.\]
Then the sequences $\left\{ f_{n}\right\} _{n\in\mathbb{Z}}$ and
$\left\{ g_{n}\right\} _{n\in\mathbb{Z}}$,

\begin{equation}
f_{n}=\mathcal{P}_{n}\,\mathfrak{F}\!\left(\left\{ \frac{\gamma_{k}^{\,2}}{\zeta_{k}}\right\} _{k=n+1}^{\infty}\right)\!,\ g_{n}=\frac{1}{w_{n-1}\mathcal{P}_{n-1}}\,\mathfrak{F}\!\left(\left\{ \frac{\gamma_{k}^{\,2}}{\zeta_{k}}\right\} _{k=-\infty}^{n-1}\right)\!,\label{eq:sols_diff_eq_inf}\end{equation}
represent two solutions of the bilateral difference equation (\ref{eq:diff_eq_inf}).

For two solutions $u=\left\{ u_{n}\right\} _{n\in\mathbb{Z}}$ and
$v=\left\{ v_{n}\right\} _{n\in\mathbb{Z}}$ of (\ref{eq:diff_eq_inf})
the Wronskian is introduced as\[
\mathcal{W}(u,v)=w_{n}\left(u_{n}v_{n+1}-u_{n+1}v_{n}\right).\]
As is well known, this is a constant independent of the index $n$.
Moreover, two solutions are linearly dependent iff their Wronskian
vanishes. For the solutions $f$ and $g$ given in (\ref{eq:sols_diff_eq_inf})
one can use (\ref{eq:F_recur_general}) to evaluate their Wronskian
getting

\[
\mathcal{W}(f,g)=\,\mathfrak{F}\!\left(\left\{ \frac{\gamma_{n}^{\,2}}{\zeta_{n}}\right\} _{n=-\infty}^{\infty}\right)\!.\]

One may also consider an application of a discrete analogue of Green's
formula to the solutions (\ref{eq:sols_diff_eq_inf}) \cite{Akhiezer}.
In general, suppose that sequences $\left\{ u_{n}\right\} _{n=0}^{\infty}$
and $\left\{ v_{n}\right\} _{n=0}^{\infty}$ solve respectively the
difference equations\begin{equation}
w_{n}u_{n+1}-\zeta_{n}^{(1)}u_{n}+w_{n-1}u_{n-1}=0,\ w_{n}v_{n+1}-\zeta_{n}^{(2)}v_{n}+w_{n-1}v_{n-1}=0,\ n\in\mathbb{N}.\label{eq:diff_eqs_12}\end{equation}
In that case it is well known and easy to check that\begin{equation}
\sum_{j=1}^{n}\left(\zeta_{j}^{(1)}-\zeta_{j}^{(2)}\right)u_{j}v_{j}=w_{0}\left(u_{0}v_{1}-u_{1}v_{0}\right)-w_{n}\left(u_{n}v_{n+1}-u_{n+1}v_{n}\right).\label{eq:sum_1n_uv}\end{equation}

\begin{proposition} \label{prop:sum_f1j_f2j} Suppose that the convergence
condition\[
\sum_{k=1}^{\infty}\left|\frac{w_{k}^{\,2}}{\zeta_{k}\zeta_{k+1}}\right|<\infty\]
is satisfied for the both difference equations in (\ref{eq:diff_eqs_12}).
Moreover, assume that\[
\sup_{n\geq1}\left|\frac{w_{n}^{\,2}}{\zeta_{n}^{(1)}\zeta_{n+1}^{(2)}}\right|<\infty\ \ \text{and}\ \ \sup_{n\geq1}\left|\frac{w_{n}^{\,2}}{\zeta_{n}^{(2)}\zeta_{n+1}^{(1)}}\right|<\infty.\]
Then the corresponding solutions $f^{(1)}$, $f^{(2)}$ from (\ref{eq:sols_diff_eq_inf})
fulfill\begin{equation}
\sum_{j=1}^{\infty}\left(\zeta_{j}^{(1)}-\zeta_{j}^{(2)}\right)f_{j}^{(1)}f_{j}^{(2)}=w_{0}\left(f_{0}^{(1)}f_{1}^{(2)}-f_{1}^{(1)}f_{0}^{(2)}\right).\label{eq:sum_inf_f1f2}\end{equation}
\end{proposition}

\begin{proof} In view of (\ref{eq:sum_1n_uv}) it suffices to show
that \begin{equation}
\lim_{n\to\infty}w_{n}f_{n}^{(1)}f_{n+1}^{(2)}=\lim_{n\to\infty}w_{n}f_{n+1}^{(1)}f_{n}^{(2)}=0.\label{eq:lim_ff_0}\end{equation}
By the convergence assumption, for all $n>n_{0}$ one has\[
|w_{n}|\leq\frac{1}{2}\,\sqrt{|\zeta_{n}^{(1)}||\zeta_{n+1}^{(1)}|}\,,\ |w_{n}|\leq\frac{1}{2}\,\sqrt{|\zeta_{n}^{(2)}||\zeta_{n+1}^{(2)}|}\,.\]
Using (\ref{eq:F_ineq_exp}), after some straightforward manipulations
one gets the estimate\begin{eqnarray*}
\left|w_{n}f_{n}^{(1)}f_{n+1}^{(2)}\right| & \leq & 2^{-2(n-n_{0})}\,\exp\!\left(\sum_{k=1}^{\infty}\left|\frac{w_{k}^{\,2}}{\zeta_{k}^{(1)}\zeta_{k+1}^{(1)}}\right|+\left|\frac{w_{k}^{\,2}}{\zeta_{k}^{(2)}\zeta_{k+1}^{(2)}}\right|\right)\prod_{k=1}^{n_{0}}\left|\frac{w_{k-1}^{\,2}}{\zeta_{k}^{(1)}\zeta_{k}^{(2)}}\right|\\
\noalign{\smallskip} &  & \times\,|\zeta_{n_{0}}^{(1)}\zeta_{n_{0}}^{(2)}|^{1/2}\,\,\frac{|w_{n}|}{\left|\zeta_{n}^{(1)}\zeta_{n+1}^{(2)}\right|^{1/2}}\,.\end{eqnarray*}
This implies (\ref{eq:lim_ff_0}). \end{proof}

In the literature on Jacobi matrices one encounters a construction
of an infinite matrix associated with the bilateral difference equation
(\ref{eq:diff_eq_inf}) \cite[\S~1.1]{Teschl}, \cite[Theorem~1.2]{Gaurschi}.
Let us define the matrix $\mathfrak{J}$ with entries $\mathfrak{J}(m,n)$,
$m,n\in\mathbb{Z}$, so that for every fixed $m$, the sequence $u_{n}=\mathfrak{J}(m,n)$,
$n\in\mathbb{Z}$, solves (\ref{eq:diff_eq_inf}) with the initial
conditions $\mathfrak{J}(m,m)=0$, $\mathfrak{J}(m,m+1)=1/w_{m}$.

Using (\ref{eq:calF_reccur_reverse}) one verifies that, for $m<n$,\[
\mathfrak{J}(m,n)=\frac{1}{w_{m}}\left(\prod_{j=m+1}^{n-1}\frac{\zeta_{j}}{w_{j}}\right)\mathfrak{F}\!\left(\frac{\gamma_{m+1}^{\,2}}{\zeta_{m+1}},\frac{\gamma_{m+2}^{\,2}}{\zeta_{m+2}},\ldots,\frac{\gamma_{n-1}^{\,2}}{\zeta_{n-1}}\right)\!.\]
Moreover, it is quite obvious that, for all $m,n\in\mathbb{Z}$,\[
\mathfrak{J}(m,n)=\frac{1}{\mathcal{W}(u,v)}\left(u_{m}v_{n}-v_{m}u_{n}\right),\]
where $\left\{ u_{n}\right\} $, $\left\{ v_{n}\right\} $ is any
couple of independent solutions of (\ref{eq:diff_eq_inf}). Hence
the matrix $\mathfrak{J}$ is antisymmetric. It also follows that,
$\forall m,n,k,\ell\in\mathbb{Z}$,\[
\mathfrak{J}(m,k)\mathfrak{J}(n,\ell)-\mathfrak{J}(m,\ell)\mathfrak{J}(n,k)=\mathfrak{J}(m,n)\mathfrak{J}(k,\ell).\]

\begin{example} As an example let us again have a look at the particular
case where $w_{n}=w$, $\zeta_{n}=\nu+n$ for all $n\in\mathbb{Z}$
and some $w,\nu\in\mathbb{C}$, $w\neq0$, $\nu\notin\mathbb{Z}$.
One finds, with the aid of (\ref{eq:BesselJ_rel_F}), that the solutions
(\ref{eq:sols_diff_eq_inf}) now read\[
f_{n}=\Gamma(\nu+1)\, w^{-\nu}J_{\nu+n}(2w),\ g_{n}=\frac{(-1)^{n}\pi}{\sin(\pi\nu)\Gamma(\nu+1)}\, w^{\nu}J_{-\nu-n}(2w).\]
Hence the Wronskian equals\[
\mathcal{W}(f,g)=\frac{\pi w}{\sin(\pi\nu)}\left(-J_{\nu}(2w)J_{-\nu-1}(2w)-J_{\nu+1}(2w)J_{-\nu}(2w)\right)=\,\mathfrak{F}\!\left(\left\{ \frac{w}{\nu+n}\right\} _{n=-\infty}^{\infty}\right)\!.\]
Recalling once more (\ref{eq:BesselJ_rel_F}) we note that the RHS
equals

\[
\lim_{N\to\infty}\mathfrak{F}\!\left(\left\{ \frac{w}{\nu-N+n}\right\} _{n=1}^{\infty}\right)=\lim_{N\to\infty}\,\sum_{n=0}^{\infty}\frac{(-1)^{n}}{n!}\,\frac{\Gamma(\nu-N+1)}{\Gamma(\nu-N+n+1)}\, w^{2n}=1.\]
Thus one gets the well known relation \cite[Eq.~9.1.15]{AbramowitzStegun}

\begin{equation}
J_{\nu+1}(2w)J_{-\nu}(2w)+J_{\nu}(2w)J_{-\nu-1}(2w)=-\frac{\sin(\pi\nu)}{\pi w}\,.\label{eq:Jnu_Jnu_eq_sin}\end{equation}

Concerning the matrix $\mathfrak{J}$, this particular choice brings
us to the case discussed in \cite[Proposition~22]{StampachStovicek}.
Then the Bessel functions $Y_{n+\nu}(2w)$ and $J_{n+\nu}(2w)$, depending
on the index $n\in\mathbb{Z}$, represent other two linearly independent
solutions of (\ref{eq:diff_eq_inf}). Since \cite[Eq.~9.1.16]{AbramowitzStegun}\[
J_{\nu+1}(z)Y_{\nu}(z)-J_{\nu}(z)Y_{\nu+1}(z)=\frac{2}{\pi z}\]
one finds that\[
\mathfrak{J}(m,n)=\pi\left(Y_{m+\nu}(2w)J_{n+\nu}(2w)-J_{m+\nu}(2w)Y_{n+\nu}(2w)\right).\]
Moreover, for $\sigma=m+\mu$ and $k=n-m>0$ one has\[
J_{\sigma+k}(2w)Y_{\sigma}(2w)-J_{\sigma}(2w)Y_{\sigma+k}(2w)=\frac{\Gamma(\sigma+k)}{\pi w^{k}\,\Gamma(\sigma+1)}\,\,\mathfrak{F}\!\left(\left\{ \frac{w}{\sigma+j}\right\} _{j=1}^{k-1}\right)\!.\]

Finally, putting $\zeta_{n}^{(1)}=\mu+n$, $\zeta_{n}^{(2)}=\nu+n$
and $w_{n}=w$, $\forall n\in\mathbb{N}$, in equation (\ref{eq:diff_eqs_12}),
one verifies that (\ref{eq:sum_inf_f1f2}) holds true and reveals
this way once more the identity (\ref{eq:sum_Bessel_mn}). \end{example}

\section{A class of Jacobi operators with point spectra \label{sec:spec_charfce}}

\subsection{The characteristic function \label{sec:pre_charfce}}

Being inspired by Proposition~\ref{prop:charfce_finJ} and notably
by equation (\ref{eq:char_pol_general}), we introduce the (renormalized)
characteristic function associated with a Jacobi matrix $\mathcal{J}$
as\begin{equation}
F_{\mathcal{J}}(z):=\mathfrak{F}\!\left(\left\{ \frac{\gamma_{n}^{\,2}}{\lambda_{n}-z}\right\} _{n=1}^{\infty}\right)\!.\label{eq:def_FJ_symm}\end{equation}
It is treated as a complex function of a complex variable $z$ and
is well defined provided the sequence in the argument of $\mathfrak{F}$
belongs to the domain $D$. Let us show that this is guaranteed under
the assumption that there exists $z_{0}\in\mathbb{C}$ such that\begin{equation}
\sum_{n=1}^{\infty}\left|\frac{w_{n}^{\,2}}{(\lambda_{n}-z_{0})(\lambda_{n+1}-z_{0})}\right|<\infty.\label{eq:assum_sum_w}\end{equation}

For $\lambda=\{\lambda_{n}\}_{n=1}^{\infty}$ let us denote\[
\mathbb{C}_{0}^{\lambda}:=\mathbb{C}\setminus\overline{\{\lambda_{n};\, n\in\mathbb{N}\}}.\]
Clearly,\[
\overline{\{\lambda_{n};\, n\in\mathbb{N}\}}=\{\lambda_{n};\, n\in\mathbb{N}\}\cup\der(\lambda)\]
where $\der(\lambda)$ stands for the set of all finite accumulation
points of the sequence $\lambda$ (i.e., $\der(\lambda)$ is equal
to the set of limit points of all possible convergent subsequences
of $\lambda$).

\begin{lemma} \label{lem:convlem_F} Let condition (\ref{eq:assum_sum_w})
be fulfilled for at least one $z_{0}\in\mathbb{C}_{0}^{\lambda}$.
Then the series \begin{equation}
\sum_{n=1}^{\infty}\frac{w_{n}^{\,2}}{(\lambda_{n}-z)(\lambda_{n+1}-z)}\label{eq:series_vw_lbd_z}\end{equation}
converges absolutely and locally uniformly in $z$ on $\mathbb{C}_{0}^{\lambda}$.
Moreover, \begin{equation}
\forall z\in\mathbb{C}_{0}^{\lambda},\quad\lim_{n\rightarrow\infty}\mathfrak{F\!}\left(\left\{ \frac{\gamma_{k}^{\,2}}{\lambda_{k}-z}\right\} _{k=1}^{n}\right)=F_{\mathcal{J}}(z),\label{eq:lim_calF_FJ}\end{equation}
and the convergence is locally uniform on $\mathbb{C}_{0}^{\lambda}$.
Consequently, $F_{\mathcal{J}}(z)$ is a well defined analytic function
on $\mathbb{C}_{0}^{\lambda}$. \end{lemma}

\begin{proof} Let $K\subset\mathbb{C}_{0}^{\lambda}$ be a compact
subset. Then the ratio\[
\frac{|\lambda_{n}-z_{0}|}{|\lambda_{n}-z|}\leq1+\frac{|z-z_{0}|}{|\lambda_{n}-z|}\]
admits an upper bound, uniform in $z\in K$ and $n\in\mathbb{N}$.
The uniform convergence on $K$ of the series (\ref{eq:series_vw_lbd_z})
thus becomes obvious.

The limit (\ref{eq:lim_calF_FJ}) follows from Lemma~\ref{lemm:F_fin_lim}.
Moreover, using (\ref{eq:lim_calF_FJ}) and also (\ref{eq:calF_reccur_reverse}),
(\ref{eq:F_ineq_exp}) one has\begin{eqnarray*}
 &  & \left|\mathfrak{F}\!\left(\left\{ \frac{\gamma_{k}^{\,2}}{\lambda_{k}-z}\right\} _{k=1}^{n}\right)-F_{\mathcal{J}}(z)\right|\leq\sum_{l=n}^{\infty}\left|\mathfrak{F}\!\left(\left\{ \frac{\gamma_{k}^{\,2}}{\lambda_{k}-z}\right\} _{k=1}^{l}\right)-\mathfrak{F}\!\left(\left\{ \frac{\gamma_{k}^{\,2}}{\lambda_{k}-z}\right\} _{k=1}^{l+1}\right)\right|\\
 &  & \leq\sum_{l=n}^{\infty}\left|\frac{w_{l}^{\,2}}{(\lambda_{l}-z)(\lambda_{l+1}-z)}\right|\exp\!\left(\sum_{k=1}^{\infty}\left|\frac{w_{k}^{\,2}}{(\lambda_{k}-z)(\lambda_{k+1}-z)}\right|\right)\!.\end{eqnarray*}
From this estimate and the locally uniform convergence of the series
(\ref{eq:series_vw_lbd_z}) one deduces the locally uniform convergence
of the sequence of functions (\ref{eq:lim_calF_FJ}). \end{proof}

By a closer inspection one finds that, under the assumptions of Lemma~\ref{lem:convlem_F},
the function $F_{\mathcal{J}}(z)$ is meromorphic on $\mathbb{C}\setminus\der(\lambda)$
with poles at the points $z=\lambda_{n}$ for some $n\in\mathbb{N}$
(not belonging to $\der(\lambda)$, however). For any such $z$, the
order of the pole is less than or equal to $r(z)$ where \[
r(z):=\sum_{k=1}^{\infty}\delta_{z,\lambda_{k}}\]
is the number of members of the sequence $\lambda$ coinciding with
$z$ (hence $r(z)=0$ for $z\in\mathbb{C}_{0}^{\lambda}$). To see
this, suppose that $r(z)\geq1$ and let $M$ be the maximal index
such that $\lambda_{M}=z$. Using (\ref{eq:F_T_recur_k}) one derives
that, for $u\in\mathbb{C}_{0}^{\lambda}$,\begin{eqnarray*}
 &  & F_{\mathcal{J}}(u)\,=\,\mathfrak{F}\!\left(\left\{ \frac{\gamma_{n}^{\,2}}{\lambda_{n}-u}\right\} _{n=1}^{M}\right)\mathfrak{F}\!\left(\left\{ \frac{\gamma_{n}^{\,2}}{\lambda_{n}-u}\right\} _{n=M+1}^{\infty}\right)\\
 &  & \phantom{F_{\mathcal{J}}(u)=}\,+\,\mathfrak{F}\!\left(\left\{ \frac{\gamma_{n}^{\,2}}{\lambda_{n}-u}\right\} _{n=1}^{M-1}\right)\frac{\gamma_{M}^{\,2}\gamma_{M+1}^{\,2}}{(u-z)(\lambda_{M+1}-u)}\,\mathfrak{F}\!\left(\left\{ \frac{\gamma_{n}^{\,2}}{\lambda_{n}-u}\right\} _{n=M+2}^{\infty}\right)\!.\end{eqnarray*}
The RHS clearly has a pole at the point $u=z$ of order at most $r(z)$.

\subsection{The Jacobi operator $J$ \label{subsec:operatorJ}}

Our goal is to investigate spectral properties of a closed operator
$J$ on $\ell^{2}(\mathbb{N})$ whose matrix in the standard basis
coincides with $\mathcal{J}$. Provided the Jacobi matrix does not
determine a bounded operator, however, there need not be a unique
way how to introduce $J$. But among all admissible operators one
may distinguish two particular cases which may respectively be regarded,
in a natural way, as the minimal and the maximal operator with the
required properties; see, for instance, \cite{Beckerman}.

\begin{definition} \label{def:Jmax_Jmin} The operator $J_{\text{max}}$
is defined so that\[
\Dom(J_{\text{max}})=\{y\in\ell^{2}(\mathbb{N});\,\mathcal{J}y\in\ell^{2}(\mathbb{N})\},\]
and one sets $J_{\text{max}}y=\mathcal{J}y$, $\forall y\in\Dom J_{\text{max}}$.
Here and in what follows $\mathcal{J}y$ is understood as the formal
matrix product while treating $y$ as a column vector. To define the
operator $J_{\text{min}}$ one first introduces the operator $\dot{J}$
so that $\Dom(\dot{J})$ is the linear hull of the standard basis,
and again $\dot{J}y=\mathcal{J}y$ for all $y\in\Dom(\dot{J})$. $\dot{J}$
is known to be closable \cite{Beckerman}, and $J_{\text{min}}$ is
defined as the closure of $\dot{J}$. \end{definition}

One has the following relations between the operators $J_{\text{min}}$,
$J_{\text{max}}$ and their adjoint operators \cite[Lemma~2.1]{Beckerman}.
Let $\mathcal{J}^{H}$ designates the Jacobi matrix obtained from
$\mathcal{J}$ by taking the complex conjugate of each entry. Then
$J_{\text{min}}^{\,\,\ast}=J_{\text{max}}^{H}$, $J_{\text{max}}^{\,\,\ast}=J_{\text{min}}^{H}$.
In particular, the maximal operator $J_{\text{max}}$ is a closed
extension of $J_{\text{min}}$. It is even true that any closed operator
$J$ whose domain contains the standard basis and whose matrix in
this basis equals $\mathcal{J}$ fulfills $J_{\text{min}}\subset J\subset J_{\text{max}}$.
Moreover, if $\mathcal{J}$ is Hermitian, i.e. $\mathcal{J}=\mathcal{J}^{H}$
(which means nothing but $\mathcal{J}$ is real), then $J_{\text{min}}^{\,\,\ast}=J_{\text{max}}\supset J_{\text{min}}$.
Hence $J_{\text{min}}$ is symmetric with the deficiency indices either
$(0,0)$ or $(1,1)$.

We are primarily interested in the situation where $J_{\text{min}}=J_{\text{max}}$
since then there exists a unique closed operator $J$ defined by the
Jacobi matrix $\mathcal{J}$, and it turns out that the spectrum of
$J$ is determined in a well defined sense by the characteristic function
$F_{\mathcal{J}}(z)$. If this happens $\mathcal{J}$ is sometimes
called proper \cite{Beckerman}.

Let us recall more details on this property. We remind the reader
that the orthogonal polynomials of the first kind, $p_{n}(z)$, are
defined by the recurrence\[
w_{n-1}p_{n-1}(z)+\lambda_{n}p_{n}(z)+w_{n}p_{n+1}(z)=z\, p_{n}(z),\quad n=1,2,3,\dots,\]
with the initial conditions $p_{0}(z)=1$, $p_{1}(z)=(z-\lambda_{1})/w_{1}$.
The orthogonal polynomials of the second kind, $q_{n}(z)$, obey the
same recurrence but the initial conditions are $q_{0}(z)=0$, $q_{1}(z)=1/w_{1}$;
see \cite{Akhiezer,Chihara}. It is not difficult to verify that these
polynomials are expressible in terms of the function $\mathfrak{F}$
as follows:\[
p_{n}(z)=\left(\prod_{k=1}^{n}\,\frac{z-\lambda_{k}}{w_{k}}\right)\mathfrak{F}\!\left(\left\{ \frac{\gamma_{l}^{\,2}}{\lambda_{l}-z}\right\} _{l=1}^{n}\right),\quad n=0,1,2\dots,\]
and\[
q_{n}(z)=\frac{1}{w_{1}}\left(\prod_{k=2}^{n}\,\frac{z-\lambda_{k}}{w_{k}}\right)\mathfrak{F}\!\left(\left\{ \frac{\gamma_{l}^{\,2}}{\lambda_{l}-z}\right\} _{l=2}^{n}\right),\quad n=1,2,3\dots.\]

The complex Jacobi matrix $\mathcal{J}$ is called determinate if
at least one of the sequences $p(0)=\{p_{n}(0)\}_{n=0}^{\infty}$
or $q(0)=\{q_{n}(0)\}_{n=0}^{\infty}$ is not an element of $\ell^{2}(\mathbb{Z}_{+})$.
For real Jacobi matrices there exits a parallel terminology. Instead
of determinate one calls $\mathcal{J}$ limit point at $+\infty$,
and instead of indeterminate one calls $\mathcal{J}$ limit circle
at $+\infty$, see \cite[p.~48]{Teschl}. According to \cite[Theorem~22.1]{Wall},
$\mathcal{J}$ is indeterminate if both $p(z)$ and $q(z)$ are elements
of $\ell^{2}$ for at least one $z\in\mathbb{C}$, and in this case
they are elements of $\ell^{2}$ for all $z\in\mathbb{C}$. For a
real Jacobi matrix $\mathcal{J}$ one can prove that it is proper
if and only if it is determinate (or, in another terminology, limit
point), see \cite[pp.~138-141]{Akhiezer} or \cite[Lemma~2.16]{Teschl}.

For complex Jacobi matrices one can also specify assumptions under
which $J_{\text{min}}=J_{\text{max}}$. In what follows, $\rho(A)$
designates the resolvent set of a closed operator $A$. Concerning
the essential spectrum, one observes that $\spec_{ess}(J_{\text{min}})=\spec_{ess}(J_{\text{max}})$
\cite[Eq.~2.10]{Beckerman}. Hence if $\rho(J_{\text{max}})\neq\emptyset$
then $\spec_{ess}(J_{\text{min}})\neq\mathbb{C}$. Moreover, in that
case $\mathcal{J}$ is determinate \cite[Theorem~2.11~(a)]{Beckerman}
and proper \cite[Theorem~2.6~(a)]{Beckerman}. This way one extracts
from \cite{Beckerman} the following result.

\begin{theorem} \label{thm:J_complex_proper} If $\rho(J_{\text{max}})\neq\emptyset$
then $J_{\text{min}}=J_{\text{max}}$. \end{theorem}

\subsection{The spectrum and the zero set of the characteristic function}

Let us define \begin{equation}
\mathfrak{Z}(\mathcal{J}):=\left\{ z\in\mathbb{C}\setminus\der(\lambda);\,\lim_{u\to z}\,(u-z)^{r(z)}F_{\mathcal{J}}(u)=0\right\} .\label{eq:def_Z}\end{equation}
Of course, $\mathfrak{Z}(\mathcal{J})\cap\mathbb{C}_{0}^{\lambda}$
is nothing but the set of zeros of $F_{\mathcal{J}}(z)$. Further,
for $k\in\mathbb{Z}_{+}$ and $z\in\mathbb{C}\setminus\der(\lambda)$
we put \begin{equation}
\xi_{k}(z):=\lim_{u\to z}\,(u-z)^{r(z)}\left(\prod_{l=1}^{k}\,\frac{w_{l-1}}{u-\lambda_{l}}\right)\!\mathfrak{F}\!\left(\left\{ \frac{\gamma_{l}^{\,2}}{\lambda_{l}-u}\right\} _{l=k+1}^{\infty}\right)\!,\label{eq:def_xi_k}\end{equation}
where one sets $w_{0}=1.$ One observes that for $k\geq M$, where
$M=M_{z}$ is either the maximal index, if any, such that $z=\lambda_{M}$,
or $M=0$ otherwise,\begin{equation}
\xi_{k}(z)=\prod_{l=1}^{k}w_{l-1}\,\Bigg(\,\prod_{\substack{l=1\\
\lambda_{l}\neq z}
}^{k}(z-\lambda_{l})\Bigg)^{\!-1}\!\mathfrak{F}\!\left(\left\{ \frac{\gamma_{l}^{\,2}}{\lambda_{l}-z}\right\} _{l=k+1}^{\infty}\right)\!.\label{eq:def_xi_k_M}\end{equation}

\begin{proposition} \label{prop:ZinSpec} Let condition (\ref{eq:assum_sum_w})
be fulfilled for at least one $z_{0}\in\mathbb{C}_{0}^{\lambda}$.
If\[
\xi_{0}(z)\equiv\lim_{u\to z}\,(u-z)^{r(z)}F_{\mathcal{J}}(u)=0\]
for some $z\in\mathbb{C}\setminus\der(\lambda)$, then $z$ is an
eigenvalue of $J_{\text{max}}$ and\[
\xi(z):=\left(\xi_{1}(z),\xi_{2}(z),\xi_{3}(z),\ldots\right)\]
is the corresponding eigenvector. \end{proposition}

\begin{proof} Using (\ref{eq:F_T_recur}) one verifies that if $\xi_{0}(z)=0$
then the column vector $\xi(z)$ solves the matrix equation $\mathcal{J}\xi(z)=z\xi(z)$.
To complete the proof one has to show that $\xi(z)$ does not vanish
and belongs to $\ell^{2}(\mathbb{N})$.

First, we claim that $\xi_{1}(z)\neq0$. Suppose, on the contrary,
that $\xi_{1}(z)=0$. Then the formal eigenvalue equation (which is
a second order recurrence) implies $\xi(z)=0$. From (\ref{eq:def_xi_k_M})
it follows that \[
\mathfrak{F}\!\left(\left\{ \frac{\gamma_{l}^{\,2}}{\lambda_{l}-z}\right\} _{l=k+1}^{\infty}\right)=0\]
for all $k\geq M=M_{z}$. This equality is in contradiction with (\ref{eq:lim_F_T_n}),
however.

Second, suppose $z\notin\der(\lambda)$ is fixed. By Lemma~\ref{lem:convlem_F},
there exists $N\in\mathbb{N}$, $N>M$, such that\[
|w_{n}^{2}|\leq|\lambda_{n}-z||\lambda_{n+1}-z|/2,\ \forall n\geq N.\]
Let us denote\[
C=\prod_{l=1}^{N}|w_{l-1}|^{2}\,\prod_{\substack{l=1\\
\lambda_{l}\neq z}
}^{N}|z-\lambda_{l}|^{-2}.\]
Using also (\ref{eq:F_ineq_exp}) one can estimate\begin{eqnarray*}
 &  & \sum_{k=N}^{\infty}|\xi_{k}(z)|^{2}\,=\,\sum_{k=N}^{\infty}\,\prod_{l=1}^{k}|w_{l-1}|^{2}\,\prod_{\substack{l=1\\
\lambda_{l}\neq z}
}^{k}|z-\lambda_{l}|^{-2}\left|\mathfrak{F}\!\left(\left\{ \frac{\gamma_{l}^{\,2}}{\lambda_{l}-z}\right\} _{l=k+1}^{\infty}\right)\right|^{2}\\
 &  & \leq\, C\exp\!\left(2\sum_{k=N+1}^{\infty}\left|\frac{w_{k}^{\,2}}{(\lambda_{k}-z)(\lambda_{k+1}-z)}\right|\right)\sum_{k=N}^{\infty}\,\prod_{l=N+1}^{k}\!\left(\frac{1}{2}\left|\frac{z-\lambda_{l-1}}{z-\lambda_{l}}\right|\right)\!.\end{eqnarray*}
Since $|\lambda_{k}-z|\geq\tau$ for all $k>M$ and some $\tau>0$,
the RHS is finite. \end{proof}

Further we wish to prove a statement converse to Proposition~\ref{prop:ZinSpec}.
Our approach is based on a formula for the Green function generalizing
a similar result known for the finite-dimensional case; see (\ref{eq:green_fin}).

\begin{proposition} \label{thm:spec_p_Z} Let condition (\ref{eq:assum_sum_w})
be fulfilled for at least one $z_{0}\in\mathbb{C}_{0}^{\lambda}$.
If $z\in\mathbb{C\setminus}\der(\lambda)$ does not belong to the
zero set $\mathfrak{Z}(\mathcal{J})$ then $z\in\rho(J_{\text{max}})$
and the Green function for the spectral parameter $z$, \[
G(z;i,j):=\langle e_{i},(J_{\text{max}}-z)^{-1}e_{j}\rangle,\ i,j\in\mathbb{N},\]
(a matrix in the standard basis) is given by the formula \begin{eqnarray}
 &  & \hskip-1.4emG(z;i,j)\,=\,-\frac{1}{w_{\max(i,j)}}\left(\prod_{l=\min(i,j)}^{\max(i,j)}\,\frac{w_{l}}{z-\lambda_{l}}\right)\label{eq:green_J}\\
\noalign{\smallskip} &  & \qquad\qquad\times\,\mathfrak{F}\!\left(\left\{ \frac{\gamma_{l}^{\,2}}{\lambda_{l}-z}\right\} _{l=1}^{\min(i,j)-1}\right)\mathfrak{F}\!\left(\left\{ \frac{\gamma_{l}^{\,2}}{\lambda_{l}-z}\right\} _{l=\max(i,j)+1}^{\infty}\right)\mathfrak{F}\!\left(\left\{ \frac{\gamma_{l}^{\,2}}{\lambda_{l}-z}\right\} _{l=1}^{\infty}\right)^{\!-1}\!\!.\nonumber \end{eqnarray}
In particular, for the Weyl m-function one has\begin{equation}
m(z):=G(z;1,1)=\frac{1}{\lambda_{1}-z}\,\mathfrak{F}\!\left(\left\{ \frac{\gamma_{l}^{\,2}}{\lambda_{l}-z}\right\} _{l=2}^{\infty}\right)\mathfrak{F}\!\left(\left\{ \frac{\gamma_{l}^{\,2}}{\lambda_{l}-z}\right\} _{l=1}^{\infty}\right)^{\!-1}\!.\label{eq:Weyl_mfce}\end{equation}
If, in addition, $|\lambda_{n}|\to\infty$ as $n\to\infty$ then $\der(\lambda)=\emptyset$
and for every $z\in\mathbb{C\setminus}\mathfrak{Z}(\mathcal{J})$,
the resolvent $(J_{\text{max}}-z)^{-1}$ is compact. \end{proposition}

\begin{proof} Denote by $R(z)_{i,j}$ the RHS of (\ref{eq:green_J}).
Thus $R(z)$ is an infinite matrix provided its entries $R(z)_{i,j}$,
$i,j\in\mathbb{N}$, make good sense. Suppose that a complex number
$z$ does not belong to $\mathfrak{Z}(J)\cup\der(\lambda)$. By Lemma~\foreignlanguage{english}{\ref{lem:convlem_F},}
in that case the RHS of (\ref{eq:green_J}) is well defined. By inspection
of the expression one finds that this is so even if $z$ happens to
coincide with a member $\lambda_{k}$ of the sequence $\lambda$ not
belonging to $\der(\lambda)$, i.e. the seeming singularity at $z=\lambda_{k}$
is removable. For the sake of simplicity we assume in the remainder
of the proof, however, that $z$ does not belong to the range of the
sequence $\lambda$. The only purpose of this assumption is just to
simplify the discussion and to avoid more complex expressions but
otherwise it is not essential for the result.

First let us show that there exists a constant $C$, possibly depending
on $z$ but independent of the indices $i$, $j$, such that\begin{equation}
|R(z)_{i,j}|\leq C\,2^{-|i-j|},\ \forall i,j\in\mathbb{N}.\label{eq:Rz_estim_exp}\end{equation}
To this end, denote \[
\tau=\inf\{|z-\lambda_{n}|;\ n\in\mathbb{N}\}>0.\]
Assuming (\ref{eq:assum_sum_w}), one can choose $n_{0}\in\mathbb{N}$
so that, for all $n\geq n_{0}$, \begin{equation}
|w_{n}|^{2}\leq|\lambda_{n}-z|\,|\lambda_{n+1}-z|/4.\label{eq:F_summand_estim}\end{equation}
Let us assume, for the sake of definiteness, that $i\leq j$. Again
by (\ref{eq:assum_sum_w}) and (\ref{eq:F_ineq_exp}),\[
\left|\mathfrak{F}\!\left(\left\{ \frac{\gamma_{l}^{\,2}}{\lambda_{l}-z}\right\} _{l=1}^{i-1}\right)\mathfrak{F}\!\left(\left\{ \frac{\gamma_{l}^{\,2}}{\lambda_{l}-z}\right\} _{l=j+1}^{\infty}\right)\mathfrak{F}\!\left(\left\{ \frac{\gamma_{l}^{\,2}}{\lambda_{l}-z}\right\} _{l=1}^{\infty}\right)^{\!-1}\right|\leq C_{1},\]
for all $i$, $j$. It remains to estimate the expression \begin{equation}
\frac{1}{|\lambda_{j}-z|}\!\left|\,\prod_{l=i}^{j-1}\,\frac{w_{l}}{\lambda_{l}-z}\right|\!.\label{eq:green_aux_prod}\end{equation}

We distinguish three cases. For the finite set of couples $i$, $j$,
$i\leq j\leq n_{0}$, (\ref{eq:green_aux_prod}) is bounded from above
by a constant $C_{2}$. Using (\ref{eq:F_summand_estim}), if $i\leq n_{0}\leq j$
then (\ref{eq:green_aux_prod}) is majorized by\begin{eqnarray*}
C_{2}\left|\frac{\lambda_{n_{0}}-z}{\lambda_{j}-z}\,\prod_{l=n_{0}}^{j-1}\,\frac{w_{l}}{\lambda_{l}-z}\right| & \leq & C_{2}\tau^{-1/2}|\lambda_{n_{0}}-z|^{1/2}\,2^{-j+n_{0}}.\end{eqnarray*}
Similarly, if $n_{0}\leq i\leq j$ then (\ref{eq:green_aux_prod})
is majorized by $\tau^{-1}2^{-j+i}$. From these partial upper bounds
the estimate (\ref{eq:Rz_estim_exp}) readily follows.

From (\ref{eq:Rz_estim_exp}) one deduces that the matrix $R(z)$
represents a bounded operator on $\ell^{2}(\mathbb{N})$. In fact,
one can write $R(z)$ as a countable sum,\begin{equation}
R(z)=\sum_{s\in\mathbb{Z}}R(z;s),\label{eq:Rz_sum_Rsz}\end{equation}
where the matrix elements of the summands are $R(z;s)_{i,j}=R(z)_{i,j}$
if $i-j=s$ and $R(z;s)_{i,j}=0$ otherwise. Thus $R(z;s)$ has nonvanishing
elements on only one parallel to the diagonal and\[
\|R(z;s)\|=\sup\{|R(z)_{i,j}|;\, i-j=s\}\leq C\,2^{-|s|}.\]
Hence the series (\ref{eq:Rz_sum_Rsz}) converges in the operator
norm. With some abuse of notation, we shall denote the corresponding
bounded operator again by the symbol $R(z)$.

Further one observes that, on the level of formal matrix products,\[
(\mathfrak{J}-z)R(z)=R(z)(\mathfrak{J}-z)=I.\]
The both equalities are in fact equivalent to the countable system
of equations (with $w_{0}=0$)\[
w_{k-1}G(z;i,k-1)+(\lambda_{k}-z)G(z;i,k)+w_{k}G(z;i,k+1)=\delta_{i,k},\ i,k\in\mathbb{N}.\]
This can be verified, in a straightforward manner, with the aid of
the rule (\ref{eq:F_T_recur_k}) or some of its particular cases (\ref{eq:F_T_recur})
and (\ref{eq:calF_reccur_reverse}). By inspection of the domains
one then readily shows that the operators $J_{\text{max}}-z$ and
$R(z)$ are mutually inverse and so $z\in\rho(J_{\text{max}})$.

Finally, suppose that $|\lambda_{n}|\to\infty$ as $n\to\infty$,
and $z\in\mathbb{C\setminus\mathfrak{Z}}(\mathcal{J})$. It turns
out that then the above estimates may be somewhat refined. In particular,
(\ref{eq:green_aux_prod}) is majorized by\[
|\lambda_{i}-z|^{-1/2}|\lambda_{j}-z|^{-1/2}\,2^{-j+i}\]
for $n_{0}\leq i,j$. But this implies that $R(z;s)_{i,j}\to0$ as
$i,j\to\infty$, with $i-j=s$ being constant. It follows that the
operators $R(z;s)$ are compact. Since the series (\ref{eq:Rz_sum_Rsz})
converges in the operator norm, $R(z)$ is compact as well.\end{proof}

\begin{corollary} \label{cor:specJ} If condition (\ref{eq:assum_sum_w})
is fulfilled for at least one $z_{0}\in\mathbb{C}_{0}^{\lambda}$
then\[
\spec(J_{\text{max}})\setminus\der(\lambda)=\spec_{p}(J_{\text{max}})\setminus\der(\lambda)=\mathfrak{Z}(\mathcal{J}).\]
 \end{corollary}

\begin{proof} Propositions~\ref{prop:ZinSpec} and\foreignlanguage{english}{
}\ref{thm:spec_p_Z} respectively imply the inclusions \[
\mathfrak{Z}(\mathcal{J})\subset\spec_{p}(J_{\text{max}})\setminus\der(\lambda),\ \ \spec(J_{\text{max}})\setminus\der(\lambda)\subset\mathfrak{Z}(\mathcal{J}).\]
This shows the equality. \end{proof}

\begin{theorem} \label{thm:specJ_proper_Znull} Suppose that the
convergence condition (\ref{eq:assum_sum_w}) is fulfilled for at
least one $z_{0}\in\mathbb{C}_{0}^{\lambda}$ and the function $F_{\mathcal{J}}(z)$
does not vanish identically on $\mathbb{C}_{0}^{\lambda}$. Then $J_{\text{min}}=J_{\text{max}}=:J$
and\begin{equation}
\spec(J)\setminus\der(\lambda)=\spec_{p}(J)\setminus\der(\lambda)=\mathfrak{Z}(\mathcal{J}).\label{eq:spec_p_Z}\end{equation}
Suppose, in addition, that the set $\mathbb{C\setminus}\der(\lambda)$
is connected. Then $\spec(J)\setminus\der(\lambda)$ consists of simple
eigenvalues which have no accumulation points in $\mathbb{C}\setminus\der(\lambda)$.
\end{theorem}

\begin{proof} By the assumptions, $\mathbb{C}\setminus(\der(\lambda)\cup\mathfrak{Z}(\mathcal{J}))\neq\emptyset$.
From Proposition~\ref{thm:spec_p_Z} one infers that $\rho(J_{\text{max}})\neq\emptyset$.
According to Theorem~\ref{thm:J_complex_proper}, one has $J_{\text{min}}=J_{\text{max}}$.
Then (\ref{eq:spec_p_Z}) becomes a particular case of Corollary~\ref{cor:specJ}.

Let us assume that $\mathbb{C\setminus}\der(\lambda)$ is connected.
Then the set $\mathbb{C}_{0}^{\lambda}$ is clearly connected as well.
Suppose on contrary that the point spectrum of $J$ has an accumulation
point in $\mathbb{C}\setminus\der(\lambda)$. Then, by equality (\ref{eq:spec_p_Z}),
the set of zeros of the analytic function $F_{\mathcal{J}}(z)$ has
an accumulation point in $\mathbb{C}\setminus\der(\lambda)$. This
accumulation point may happen to be a member $\lambda_{n}$ of the
sequence $\lambda$, but then one knows that $F_{\mathcal{J}}(z)$
has a pole of finite order at $\lambda_{n}$. In any case, taking
into account that $\mathbb{C}_{0}^{\lambda}$ is connected one comes
to the conclusion that $F_{\mathcal{J}}(z)=0$ everywhere on $\mathbb{C}_{0}^{\lambda}$,
a contradiction. \end{proof}

\begin{remark} \label{rem:Fjz_0_halfplane} Theorem~\ref{thm:specJ_proper_Znull}
is derived under two assumptions:

\smallskip\noindent (i)~The convergence condition (\ref{eq:assum_sum_w})
is fulfilled for at least one $z_{0}\in\mathbb{C}_{0}^{\lambda}$.

\smallskip\noindent (ii)~The function $F_{\mathcal{J}}(z)$ does
not vanish identically on $\mathbb{C}_{0}^{\lambda}$.

\noindent But let us point out that assumption (ii) is automatically
fulfilled if (i) is true and the range of the sequence $\lambda$
is contained in a halfplane. This happens, for example, if the sequence
$\lambda$ is real or the sequence $\{\Re\lambda_{n}\}_{n=1}^{\infty}$
is semibounded. In fact, let us for definiteness consider the latter
case and suppose that $\Re\lambda_{n}\geq c$, $\forall n\in\mathbb{N}$.
Then $(-\infty,c)\subset\mathbb{C}_{0}^{\lambda}$ and $1/|\lambda_{n}-z|$
tends to $0$ monotonically for all $n$ as $z\to-\infty$. Similarly
as in (\ref{eq:F_ineq_exp}) one derives the estimate\[
\left|F_{\mathcal{J}}(z)-1\right|\leq\exp\!\left(\sum_{n=1}^{\infty}\left|\frac{w_{n}^{\,2}}{(\lambda_{n}-z)(\lambda_{n+1}-z)}\right|\right)-1.\]
It follows that $\lim_{z\to-\infty}F_{\mathcal{J}}(z)=1$. Notice
that in the real case, the function $F_{\mathcal{J}}(z)$ can identically
vanish neither on the upper nor on the lower halfplane. \end{remark}

\begin{corollary} \label{cor:specJ_proper_Znull_real} Let $\mathcal{J}$
be real and suppose that (\ref{eq:assum_sum_w}) is fulfilled for
at least one $z_{0}\in\mathbb{C}_{0}^{\lambda}$. Then $J_{\text{min}}=J_{\text{max}}=J$
is self-adjoint and $\spec(J)\setminus\der(\lambda)=\mathfrak{Z}(\mathcal{J})$
consists of simple real eigenvalues which have no accumulation points
in $\mathbb{R}\setminus\der(\lambda)$. \end{corollary}

\begin{proof} Some assumptions in Theorem~\ref{thm:specJ_proper_Znull}
become superfluous if $\mathcal{J}$ is real. As observed in Remark~\ref{rem:Fjz_0_halfplane},
assuming the convergence condition the function $F_{\mathcal{J}}(z)$
cannot vanish identically on $\mathbb{C}_{0}^{\lambda}$. The operator
$J$ is self-adjoint and may have only real eigenvalues. The set $\mathbb{C}\setminus\der(\lambda)$
may happen to be disconnected only if the range of the sequence $\lambda$
is dense in $\mathbb{R}$, i.e. $\der(\lambda)=\mathbb{R}$. But even
then the conclusion of the theorem remains trivially true. \end{proof}

Let us complete this analysis by a formula for the norms of the eigenvectors
described in Proposition~\ref{prop:ZinSpec}. In order to simplify
the discussion we restrict ourselves to the domain $\mathbb{C}_{0}^{\lambda}$.
Then instead of (\ref{eq:def_xi_k}) one may write\begin{equation}
\xi_{k}(z)=\left(\prod_{l=1}^{k}\,\frac{w_{l-1}}{z-\lambda_{l}}\right)\mathfrak{F}\!\left(\left\{ \frac{\gamma_{l}^{\,2}}{\lambda_{l}-z}\right\} _{l=k+1}^{\infty}\right)\!,\ z\in\mathbb{C}_{0}^{\lambda},\ k\in\mathbb{Z}_{+}.\label{eq:def_xik_new}\end{equation}
This is in fact nothing but the solution $f_{n}$ from (\ref{eq:sols_diff_eq_inf})
restricted to nonnegative indices.

\begin{proposition} \label{prop:norm_xi} If $z\in\mathbb{C}_{0}^{\lambda}$
satisfies (\ref{eq:assum_sum_w}) then the functions $\xi_{k}(z)$,
$k\in\mathbb{Z}_{+}$, defined in (\ref{eq:def_xik_new}) fulfill
\begin{equation}
\sum_{k=1}^{\infty}\xi_{k}(z)^{2}=\xi'_{0}(z)\xi_{1}(z)-\xi_{0}(z)\xi'_{1}(z).\label{eq:xi_sum}\end{equation}
Particularly, if in addition $\mathcal{J}$ is real and $z\in\mathbb{R}\cap\mathbb{C}_{0}^{\lambda}$
is an eigenvalue of $J$ then $\xi(z)=(\xi_{k}(z))_{k=1}^{\infty}$
is a corresponding eigenvector and \begin{equation}
\|\xi(z)\|^{2}=\xi'_{0}(z)\xi_{1}(z).\label{eq:eigenvec_norm}\end{equation}
\end{proposition}

\begin{proof} Put $\zeta_{j}^{(1)}=z-\lambda_{j}$, $\zeta_{j}^{(2)}=y-\lambda_{j}$,
$j\in\mathbb{N}$, in equation (\ref{eq:sum_1n_uv}), where $z,y\in\mathbb{C}_{0}^{\lambda}$.
Then Proposition~\ref{prop:sum_f1j_f2j} is applicable to $f_{j}^{(1)}=\xi_{j}(z)$,
$f_{j}^{(2)}=\xi_{j}(y)$, $j\in\mathbb{Z}_{+}$. Hence ($w_{0}=1$)\[
(z-y)\sum_{k=0}^{\infty}\xi_{k}(z)\xi_{k}(y)=\xi_{1}(z)\xi_{0}(y)-\xi_{0}(z)\xi_{1}(y).\]
Now the limit $y\to z$ can be treated in a routine way. \end{proof}

\begin{corollary} Suppose $\mathcal{J}$ is real and let condition
(\ref{eq:assum_sum_w}) be fulfilled for at least one $z_{0}\in\mathbb{C}_{0}^{\lambda}$.
Then the function $F_{\mathcal{J}}(z)$ has only simple real zeros
on $\mathbb{C}_{0}^{\lambda}$. \end{corollary}

\begin{proof} Suppose $z\in\mathbb{C}_{0}^{\lambda}$ is a zero of
$F_{\mathcal{J}}(z)$, i.e. $F_{\mathcal{J}}(z)\equiv\xi_{0}(z)=0$.
Then $z$ is a real eigenvalue of $J$ where $J=J_{\text{max}}=J_{\text{min}}$
is self-adjoint, as we know from Corollary~\ref{cor:specJ_proper_Znull_real}.
Moreover, by Proposition~\ref{prop:ZinSpec}, $\xi(z)\neq0$ is a
corresponding real eigenvector. Hence from (\ref{eq:eigenvec_norm})
one infers that necessarily $\xi_{0}'(z)\neq0$. \end{proof}

\subsection{Approximation of the spectrum by spectra of truncated matrices}

Let us first introduce some additional notation which will be needed
in the current subsection. $W$ and $L$ stand for the diagonal matrix
operators whose diagonals equal $w$ and $\lambda$, respectively.
$U$ designates the unilateral shift and $U^{*}$ its adjoint operator
($Ue_{n}=e_{n+1}$ for $n=1,2,3,\dots$, with $e_{n}$ being the vectors
of the standard basis).

For a Jacobi matrix $\mathcal{J}$ we introduce the set \begin{equation}
\Lambda(\mathcal{J}):=\{\mu\in\mathbb{C};\,\lim_{n\rightarrow\infty}\dist(\spec(J_{n}),\mu)=0\}\label{eq:def_Lambda}\end{equation}
where $J_{n}$ is defined in (\ref{eq:def Jn}) but now with $v=w$.
Thus $\mu\in\Lambda(\mathcal{J})$ iff there exists a sequence of
eigenvalues $\{\mu_{n}\}$ of $J_{n}$ such that $\mu_{n}\rightarrow\mu$
as $n\rightarrow\infty$.

\begin{remark} (i)~Definition (\ref{eq:def_Lambda}) is taken, for
example, from \cite{Arveson}. Some authors prefer to work, however,
with the set $\tilde{\Lambda}(\mathcal{J})$ formed by all limit points
of sequences $\{\mu_{k}\}$ of eigenvalues from $\spec(J_{n_{k}})$,
with $\{n_{k}\}$ being any possible strictly increasing sequence
of indices \cite{HartmanWinter,Ifantis_etal}. One clearly has $\phantom{\Big|}$$\Lambda(\mathcal{J})\subset\tilde{\Lambda}(\mathcal{J})$
and the inclusion is in general strict as demonstrated by a simple
example in \cite[Eq.~(2)]{HartmanWinter}.

(ii)~In the case where the sequences $\lambda$ and $w$ are real
and positive, respectively, the relation between the spectrum of a
Jacobi operator $J$ and the set $\Lambda(\mathcal{J})$ has been
intensively studied by several authors. Most results of this kind
are restricted to the bounded case, however. The inclusion $\spec(J)\subset\Lambda(\mathcal{J})$
is proved for bounded Jacobi matrices in \cite[Theorem~2.3]{Arveson}.
It cannot be replaced, in general, by the sign of equality as demonstrates
an counterexample constructed in the Appendix in \cite{Arveson}.

(iii)~In \cite[Theorem~5.1]{IfantisSiafarikas} it is shown that
$\spec(J_{S})\subset\Lambda(\mathcal{J})$ where the operator\[
J_{S}:=L+WU^{*}+UW\]
(an operator sum) is assumed to be self-adjoint. Some reasoning used
below, notably that in Lemmas~\ref{lem:JS_strong_conv} and \ref{prop:spec_in_Lambda},
is inspired by this article. We also remark that Theorem~\ref{thm:spec_eq_Lbd}
below partially reproduces and overlaps with Theorems~2.1 and 2.10
from \cite{Ifantis_etal}. \end{remark}

\begin{lemma} \label{prop:Lambda_sub_Z} If condition (\ref{eq:assum_sum_w})
is fulfilled for at least one $z_{0}\in\mathbb{C}_{0}^{\lambda}$
then \begin{equation}
\Lambda(\mathcal{J})\setminus\der(\lambda)\subset\mathfrak{Z}(\mathcal{J}).\label{eq:Lambda_in_Z}\end{equation}
\end{lemma}

\begin{proof} Let $z\in\Lambda(\mathcal{J})\setminus\der(\lambda)$.
Then, by definition, there exists a sequence $\{\mu_{n}\}$ such that
$\mu_{n}\in\spec(J_{n})$ and $\mu_{n}\rightarrow z$ as $n\rightarrow\infty$.
Considering sufficiently large indices $n$ one may assume that $\mu_{n}$
does not coincide with any member $\lambda_{k}$ of the sequence $\lambda$,
possibly except of $z$. Using (\ref{eq:char_pol_general}) one gets,
for all sufficiently large $n$, \begin{equation}
0=\det(\mu_{n}I_{n}-J_{n})=\left(\prod_{\substack{k=1\\
\lambda_{k}\neq z}
}^{n}(\mu_{n}-\lambda_{k})\right)\lim_{u\to\mu_{n}}\,(u-z)^{r(z)}\,\mathfrak{F}\!\left(\left\{ \frac{\gamma_{l}^{\,2}}{\lambda_{l}-u}\right\} _{l=1}^{n}\right)\!.\label{eq:lim_mu_z}\end{equation}
Denote temporarily\[
\tilde{F}^{n}(u)=(u-z)^{r(z)}\,\mathfrak{F}\!\left(\left\{ \frac{\gamma_{l}^{\,2}}{\lambda_{l}-u}\right\} _{l=1}^{n}\right)\!,\ u\in\mathbb{C}_{0}^{\lambda},\ n\in\mathbb{N}\cup\{\infty\}.\]
All functions $\tilde{F}^{n}(u)$ have a removable singularity at
$u=z$. Moreover, with some slight modification of Lemma~\ref{lem:convlem_F}
one can show that $\tilde{F}^{n}(u)\to\tilde{F}^{\infty}(u)$ as $n\to\infty$
uniformly on a neighborhood of the point $z$. Since all the involved
functions are analytic on this domain, one knows that the corresponding
derivatives converge uniformly as well. Equation (\ref{eq:lim_mu_z})
means that $\tilde{F}^{n}(\mu_{n})=0$ for all $n$ sufficiently large.
One concludes that\[
\lim_{u\to z}\,(u-z)^{r(z)}\, F_{\mathcal{J}}(u)=\tilde{F}^{\infty}(z)=\lim_{n\to\infty}\,\tilde{F}^{n}(\mu_{n})=0.\]
Hence $z\in\mathfrak{Z}(\mathcal{J})$. \end{proof}

Denote by $P_{n}$, $n\in\mathbb{N}$, the orthogonal projection on
$\ell^{2}(\mathbb{N})$ onto the linear hull of the first $n$ vectors
of the standard basis. Then the finite-dimensional operator $J_{n}$
introduced in (\ref{eq:def Jn}) can be identified with $P_{n}JP_{n}$
restricted to the subspace $\Ran P_{n}$. Here $J$ is any operator
such that $\dot{J}\subset J$ (see Definition~\ref{def:Jmax_Jmin}).

\begin{lemma} \label{lem:JS_strong_conv} Let $J$ be any operator
on $\ell^{2}(\mathbb{N})$ fulfilling $\dot{J}\subset J$. Then (in
$\ell^{2}(\mathbb{N})$) \[
\lim_{n\rightarrow\infty}P_{n}JP_{n}x=Jx,\ \forall x\in\Dom(J)\cap\Dom(WU^{\ast}).\]
\end{lemma}

\begin{proof} Let $x\in\Dom(J)\cap\Dom(WU^{\ast})$. Since\[
(J-P_{n}JP_{n})x=(I-P_{n})Jx+P_{n}J(I-P_{n})x\]
one has\[
\|(J-P_{n}JP_{n})x\|^{2}=\|(I-P_{n})Jx\|^{2}+|w_{n}|^{2}|\langle x,e_{n+1}\rangle|^{2}=\|(I-P_{n})Jx\|^{2}+|\langle WU^{*}x,e_{n}\rangle|^{2}.\]
The RHS clearly tends to zero as $n\rightarrow\infty$. \end{proof}

\begin{lemma} \label{prop:spec_in_Lambda} Suppose that $\mathcal{J}$
is real and let $J$ be a self-adjoint extension of the symmetric
operator $J_{\text{min}}$. If $\Dom(J)\subset\Dom(WU^{\ast})$ then
$\spec(J)\subset\Lambda(\mathcal{J})$. If, in addition, condition
(\ref{eq:assum_sum_w}) is fulfilled for at least one $z_{0}\in\mathbb{C}_{0}^{\lambda}$
then \begin{equation}
\mathfrak{Z}(\mathcal{J})=\spec(J)\setminus\der(\lambda)=\spec_{p}(J)\setminus\der(\lambda)=\Lambda(\mathcal{J})\setminus\der(\lambda).\label{eq:Z_specpoint_Lambda}\end{equation}
 \end{lemma}

\begin{proof} First suppose that, on the contrary, $\mu\in\spec(J)$
and $\mu\notin\Lambda(\mathcal{J})$. Then there exist $d>0$ and
a subsequence $\{J_{n_{k}}\}_{k=1}^{\infty}$ such that $\dist(\mu,\spec J_{n_{k}})\geq d$
for all $k$. Let $x$ be an arbitrary vector from $\Dom(J)\subset\Dom(WU^{\ast})$.

As is well known and easy to verify, if $A$ is a self-adjoint operator
such that $0$ belongs to its resolvent set then $\|A^{-1}\|\leq1/\dist(0,\spec A)$.
This implies that\[
\forall f\in\Dom A,\,\|Af\|\geq\dist(0,\spec A)\|f\|.\]
Applying this observation to the operators $J_{n_{k}}$ one gets \[
\|(\mu P_{n_{k}}-P_{n_{k}}JP_{n_{k}})x\|\geq d\|P_{n_{k}}x\|.\]
Sending $k\to\infty$ and referring to Lemma~\ref{lem:JS_strong_conv}
one concludes that $\|(\mu-J)x\|\geq d\|x\|$, for every $x\in\Dom(J)$.
By the Weyl criterion $\mu\in\rho(J)$, a contradiction.

If, in addition, (\ref{eq:assum_sum_w}) is fulfilled for at least
one $z_{0}\in\mathbb{C}_{0}^{\lambda}$ then we know from Corollary~\ref{cor:specJ_proper_Znull_real}
that $J$ is in fact unique and the first two equalities in (\ref{eq:Z_specpoint_Lambda})
are true. So combination of Lemma~\ref{prop:Lambda_sub_Z} with the
first part of the current lemma yields\[
\mathfrak{Z}(\mathcal{J})=\spec(J)\setminus\der(\lambda)\subset\Lambda(\mathcal{J})\setminus\der(\lambda)\subset\mathfrak{Z}(\mathcal{J}).\]
This completes the proof. \end{proof}

\begin{theorem} \label{thm:spec_eq_Lbd} Suppose that the Jacobi
matrix $\mathcal{J}$ is real and let condition (\ref{eq:assum_sum_w})
be fulfilled for at least one $z_{0}\in\mathbb{C}_{0}^{\lambda}$.
If any of the following two conditions is satisfied:

\smallskip(i)~$\lim_{n\rightarrow\infty}w_{n}=0$,

\smallskip(ii)~$\lim_{n\rightarrow\infty}|\lambda_{n}|=\infty$
and\begin{equation}
\limsup_{n\to\infty}\frac{|w_{n}|}{|\lambda_{n}|}\,+\,\limsup_{n\to\infty}\frac{|w_{n}|}{|\lambda_{n+1}|}\,<1,\label{eq:limsup_w_lbd}\end{equation}
then \begin{equation}
\Lambda(\mathcal{J})=\spec(J),\label{eq:Lambda_eq_spec}\end{equation}
where $J$ is the unique self-adjoint operator whose matrix in the
standard basis equals $\mathcal{J}$. \end{theorem}

\begin{proof} According to Corollary~\ref{cor:specJ_proper_Znull_real},
$J=J_{\text{min}}=J_{\text{max}}$ is the unique self-adjoint operator
determined by $\mathcal{J}$.

(i) Let $\lim_{n\rightarrow\infty}w_{n}=0$. Then we claim that\begin{equation}
\der(\lambda)=\spec_{ess}(J)\subset\spec(J)\subset\Lambda(\mathcal{J}).\label{eq:aux_speccssJ_specJ_Lambda}\end{equation}
In fact, the Hermitian operator $WU^{*}+UW$ is compact and\begin{equation}
J=L+(WU^{*}+UW)\label{eq:J_eq_LplusWW}\end{equation}
(an operator sum, $\Dom J=\Dom L$). But the Weyl theorem tells us
that $J$ has the same essential spectrum as $L$, which is nothing
but $\der(\lambda)$. The last inclusion in (\ref{eq:aux_speccssJ_specJ_Lambda})
follows from Lemma~\ref{prop:spec_in_Lambda}. Again by Lemma~\ref{prop:spec_in_Lambda},
$\spec(J)\setminus\der(\lambda)=\Lambda(\mathcal{J})\setminus\der(\lambda)$,
which jointly with (\ref{eq:aux_speccssJ_specJ_Lambda}) implies (\ref{eq:Lambda_eq_spec}).

(ii) Suppose that $|\lambda_{n}|\rightarrow\infty$ and (\ref{eq:limsup_w_lbd})
is true. Then clearly $\der(\lambda)=\emptyset$. We wish to verify
the assumptions of Lemma~\ref{prop:spec_in_Lambda}. Since the sum
of a self-adjoint operator with a bounded Hermitian operator does
not change the domain one may suppose, without loss of generality,
that\[
|\lambda_{n}|\geq1,\ \forall n,\ \,\text{and}\,\ \sup_{n}\frac{|w_{n}|}{|\lambda_{n}|}+\sup_{n}\frac{|w_{n}|}{|\lambda_{n+1}|}\leq\alpha<1,\]
where $\alpha$ is a constant. But this means that\[
\|(WU^{*}+UW)L^{-1}\|\leq\|WU^{*}L^{-1}\|+\|WL^{-1}\|\leq\alpha,\]
and hence $(WU^{*}+UW)$ is $L$-bounded with a relative bound smaller
than $1$. By the Kato-Rellich theorem \cite[Theorem~V.4.3]{Kato},
(\ref{eq:J_eq_LplusWW}) is again true. Moreover, \[
\Dom J=\Dom L\subset\Dom(WU^{\ast}).\]
Hence in this case, too, Lemma~\ref{prop:spec_in_Lambda} implies
(\ref{eq:Lambda_eq_spec}). \end{proof}

\begin{remark} Observe that (\ref{eq:assum_sum_w}) already implies
that\[
\lim_{n\to\infty}\left|\frac{w_{n}^{\,2}}{\lambda_{n}\lambda_{n+1}}\right|=0.\]
Thus assumption (\ref{eq:limsup_w_lbd}) may well turn out to be superfluous
in some concrete examples. \end{remark}

\section{Examples \label{sec:example}}

\subsection{Explicitly solvable examples of point spectra}

In all examples presented below the Jacobi matrix $\mathcal{J}$ is
real and symmetric. The set of accumulation points $\der(\lambda)$
is either empty or the one-point set $\{0\}$. Moreover, condition
(\ref{eq:assum_sum_w}) is readily checked to be satisfied for any
$z_{0}\in\mathbb{C}\setminus\mathbb{R}$. Thus Corollary~\ref{cor:specJ_proper_Znull_real}
applies to all these examples and may be used to determine the spectrum
of the unique self-adjoint operator $J$ determined by $\mathcal{J}$
(recall also definition (\ref{eq:def_Z}) of the zero set of the characteristic
function). In addition, Proposition~\ref{prop:ZinSpec} and equation
(\ref{eq:def_xi_k}) (or (\ref{eq:def_xik_new})) provide us with
explicit formulas for the corresponding eigenvectors.

\begin{example} \label{example:1_lindag} This is an example of an
unbounded Jacobi operator. Let $\lambda_{n}=n\alpha$, where $\alpha\in\mathbb{R}\setminus\{0\}$,
and $w_{n}=w>0$ for all $n\in\mathbb{N}$. Thus \[
J=\begin{pmatrix}\alpha & w\\
w & 2\alpha & w\\
 & w & 3\alpha & w\\
 &  & \ddots & \ddots & \ddots\end{pmatrix}\!.\]
One has $\der(\lambda)=\emptyset$ and $\spec(J)=\mathfrak{Z}(\mathcal{J})$.
Using (\ref{eq:BesselJ_rel_F}) one derives that \[
\mathfrak{F}\!\left(\!\left\{ \frac{\gamma_{k}^{2}}{\alpha k-z}\right\} _{k=r+1}^{\infty}\right)\!=\mathfrak{F}\!\left(\!\left\{ \frac{w}{\alpha k-z}\right\} _{k=r+1}^{\infty}\right)\!=\left(\frac{w}{\alpha}\right)^{-r+z/\alpha}\Gamma\!\left(1+r-\frac{z}{\alpha}\right)\! J_{r-z/\alpha}\!\left(\frac{2w}{\alpha}\right)\]
for $r\in\mathbb{Z}_{+}$. It follows that\begin{equation}
\spec(J)=\left\{ z\in\mathbb{C};\, J_{-z/\alpha}\!\left(\frac{2w}{\alpha}\right)=0\right\} .\label{eq:specJ_BesselJ}\end{equation}
For the corresponding eigenvectors $v(z)$ one obtains \[
v_{k}(z)=(-1)^{k}J_{k-z/\alpha}\!\left(\frac{2w}{\alpha}\right)\!,\ k\in\mathbb{N}.\]

Let us remark that the characterization of the spectrum of $J$, as
given in (\ref{eq:specJ_BesselJ}), was observed earlier by several
authors, see \cite[Sec.~3]{Ikebe_atal} and \cite[Thm.~3.1]{Petropoulou_etal}.
We discuss in more detail solutions of the characteristic equation
$J_{-z}(2w)=0$ below in Subsection~\ref{subsec:example_lin_diag}.
\end{example}

Further we describe four examples in which the Jacobi matrix always
represents a compact operator on $\ell^{2}(\mathbb{N})$. We shall
make use of the following construction. Let us fix positive constants
$c$, $\alpha$ and $\beta$. For $n\in\mathbb{Z}_{+}$ we define
the $c$-deformed number $n$ as\[
[n]_{c}=\sum_{i=0}^{n-1}c^{i}.\]
Hence $[n]_{c}=(c^{n}-1)/(c-1)$ if $c\neq1$ and $[n]_{c}=n$ for
$c=1$. Notice that \begin{equation}
\frac{[n+m-1]_{c}-[n-1]_{c}}{[m]_{c}}=[n]_{c}-[n-1]_{c},\ \,\forall n,m\in\mathbb{N}.\label{eq:const_dist}\end{equation}
As for the Jacobi matrix $\mathcal{J}$, we put

\begin{equation}
\lambda_{n}=\frac{1}{\alpha+[n-1]_{c}}\,,\ w_{n}=\beta\sqrt{\lambda_{n}-\lambda_{n+1}}\,,\quad n=1,2,3,\dots.\label{eq:ex_compact_lbd_w}\end{equation}

We claim that the identity\begin{eqnarray}
 &  & \hskip-2em\sum_{k_{1}=r}^{\infty}\,\sum_{k_{2}=k_{1}+2}^{\infty}\!\,\dots\!\,\sum_{k_{s}=k_{s-1}+2}^{\infty}\nonumber \\
\noalign{\smallskip} &  & \qquad\quad\times\,\frac{w_{k_{1}}^{\,2}}{(\lambda_{k_{1}}-z)(\lambda_{k_{1}+1}-z)}\frac{w_{k_{2}}^{\,2}}{(\lambda_{k_{2}}-z)(\lambda_{k_{2}+1}-z)}\cdots\frac{w_{k_{s}}^{\,2}}{(\lambda_{k_{s}}-z)(\lambda_{k_{s}+1}-z)}\nonumber \\
\noalign{\medskip} &  & \hskip-2em=\,\frac{(-1)^{s}}{z^{s}}\,\beta^{2s}\prod_{i=1}^{s}\left([\, i\,]_{c}\!\left(1-\frac{z}{\lambda_{r+i-1}}\right)\right)^{\!-1}\label{eq:sumF_const_dist}\end{eqnarray}
holds for every $r,s\in\mathbb{N}$. In fact, to show (\ref{eq:sumF_const_dist})
one can proceed by mathematical induction in $s$. The case $s=1$
as well as all induction steps are straightforward consequences of
the equality\begin{eqnarray*}
 &  & \frac{w_{j}^{\,2}}{(\lambda_{j}-z)(\lambda_{j+1}-z)}\,\prod_{i=1}^{s-1}\left([\, i\,]_{c}\!\left(1-\frac{z}{\lambda_{j+i+1}}\right)\right)^{\!-1}\\
\noalign{\smallskip} &  & =\,-\frac{\beta^{2}}{z}\left(\,\prod_{i=1}^{s}\left([\, i\,]_{c}\!\left(1-\frac{z}{\lambda_{j+i-1}}\right)\right)^{\!-1}-\prod_{i=1}^{s}\left([\, i\,]_{c}\!\left(1-\frac{z}{\lambda_{j+i}}\right)\right)^{\!-1}\right)\!,\end{eqnarray*}
with $s=1,2,3,\ldots$, which in turn can be verified with the aid
of (\ref{eq:const_dist}).

A principal consequence one can draw from (\ref{eq:sumF_const_dist})
is that \begin{equation}
\mathfrak{F}\!\left(\left\{ \frac{\gamma_{n}^{\,2}}{\lambda_{n}-z}\right\} _{n=r+1}^{\infty}\right)=\sum_{s=0}^{\infty}\,\frac{\beta^{2s}}{z^{s}}\,\prod_{i=1}^{s}\left([\, i\,]_{c}\!\left(1-\frac{z}{\lambda_{r+i}}\right)\right)^{\!-1}\label{eq:F_T_r_spec_comp}\end{equation}
holds for all $r\in\mathbb{Z}_{+}$.

\begin{example} \label{example:2_1over_n} In (\ref{eq:ex_compact_lbd_w}),
let us put $c=1$ and $\alpha=1$ while $\beta$ is arbitrary positive.
Then $\lambda_{n}=1/n$, $w_{n}=\beta/\sqrt{n(n+1)}$~, for all $n\in\mathbb{N}$,
and so\[
J=\begin{pmatrix}1 & \beta/\sqrt{2}\\
\beta/\sqrt{2} & 1/2 & \beta/\sqrt{6}\\
 & \beta/\sqrt{6} & 1/3 & \beta/\sqrt{12}\\
 &  & \ddots & \ddots & \ddots\end{pmatrix}\!.\]
One finds that\begin{eqnarray}
\mathfrak{F}\!\left(\left\{ \frac{\gamma_{n}^{\,2}}{\lambda_{n}-z}\right\} _{n=k+1}^{\infty}\right) & = & \sum_{s=0}^{\infty}\,\frac{\beta^{2s}}{s!\, z^{s}}\,\prod_{j=1}^{s}\frac{1}{1-(k+j)z}\,=\,{}_{0}F_{1}\!\left(k+1-\frac{1}{z};-\frac{\beta^{2}}{z^{2}}\right)\nonumber \\
\noalign{\smallskip} & = & \left(\frac{z}{\beta}\right)^{\! k-1/z}\Gamma\!\left(k+1-\frac{1}{z}\right)\! J_{k-1/z}\!\left(\frac{2\beta}{z}\right)\!,\label{eq:excomp_aux_F_Bessel}\end{eqnarray}
with $k\in\mathbb{Z}_{+}$, see \cite[Eq.~9.1.69]{AbramowitzStegun}.
Then \[
F_{\mathcal{J}}(z)=\Gamma\!\left(1-\frac{1}{z}\right)\left(\frac{z}{\beta}\right)^{\!-1/z}J_{-1/z}\!\left(\frac{2\beta}{z}\right)\]
and\[
\spec(J)=\left\{ z\in\mathbb{R}\setminus\{0\};\, J_{-1/z}\!\left(\frac{2\beta}{z}\right)=0\right\} \cup\{0\}.\]
For the corresponding eigenvectors $v(z)$ one has\[
v_{k}(z)=\sqrt{k}\, z^{\!-1/z}J_{k-1/z}\!\left(\frac{2\beta}{z}\right)\!,\ k\in\mathbb{N}.\]
\end{example}

\begin{example} \label{example:3_power_q} Now we suppose in (\ref{eq:ex_compact_lbd_w})
that $c>1$ and put $\alpha=1/(c-1)$. Then $\lambda_{n}=(c-1)c^{-n+1}$
and $w_{n}=\beta\,(c-1)c^{-1/2}c^{(-n+1)/2}$. In order to simplify
the expressions let us divide all matrix elements by the term $c-1$.
Furthermore, we also replace the parameter $\beta$ by $\beta c^{1/2}$
and use the substitution $c=1/q$, with $0<q<1$. Thus for the matrix
simplified in this way we have $\lambda_{n}=q^{n-1}$ and $w_{n}=\beta q^{(n-1)/2}$.
Hence\[
J=\begin{pmatrix}1 & \beta\\
\beta & q & \beta\sqrt{q}\\
 & \beta\sqrt{q} & q^{2} & \beta q\\
 &  & \ddots & \ddots & \ddots\end{pmatrix}\!.\]
Using (\ref{eq:F_geom_tw}), equation (\ref{eq:F_T_r_spec_comp})
then becomes\begin{eqnarray*}
\mathfrak{F}\!\left(\left\{ \frac{\gamma_{n}^{\,2}}{\lambda_{n}-z}\right\} _{n=r+1}^{\infty}\right) & = & \sum_{s=0}^{\infty}(-1)^{s}\frac{q^{s(s-1)}}{(q;q)_{s}(q^{r}/z;q)_{s}}\left(\frac{q^{r}\beta}{z}\right)^{\!2s}\\
\noalign{\smallskip} & = & _{0}\phi_{1}(;q^{r}/z;q,-q^{r}\beta^{2}/z^{2}),\ r\in\mathbb{Z}_{+}.\end{eqnarray*}
Thus we get\[
\spec(J)=\left\{ z\in\mathbb{R}\setminus\{0\};\,\Big(\frac{1}{z}\,;q\Big)_{\!\infty}\,\,{}_{0}\phi_{1}\!\!\left(;\frac{1}{z};q,-\frac{\beta^{2}}{z^{2}}\right)=0\right\} \cup\{0\}.\]
The $k$th entry of an eigenvector $v(z)$ corresponding to a nonzero
point of the spectrum $z$ may be written in the form\[
v_{k}(z)=q^{(k-1)(k-2)/4}\left(\frac{\beta}{z}\right)^{\! k-1}\Big(\frac{q^{k}}{z}\,;q\Big)_{\!\infty}\,\,{}_{0}\phi_{1}\!\!\left(;\frac{q^{k}}{z};q,-\frac{q^{k}\beta^{2}}{z^{2}}\right)\!,\ k\in\mathbb{N}.\]
\end{example}

Further we shortly discuss two examples of Jacobi matrices with zero
diagonal and $w\in\ell^{2}(\mathbb{N})$. Such a Jacobi matrix represents
a compact operator (even Hilbert-Schmidt). The characteristic function
is an even function,\[
\mathfrak{F}\!\left(\left\{ \frac{\gamma_{n}^{\,2}}{z}\right\} _{n=1}^{\infty}\right)=\sum_{m=0}^{\infty}\frac{(-1)^{m}}{z^{2m}}\sum_{k_{1}=1}^{\infty}\,\sum_{k_{2}=k_{1}+2}^{\infty}\,\dots\,\sum_{k_{m}=k_{m-1}+2}^{\infty}w_{k_{1}}^{\,2}w_{k_{2}}^{\,2}\dots w_{k_{m}}^{\,2}.\]
Hence the spectrum of $J$ is symmetric with respect to the origin.

Though $0$ always belongs to the spectrum of a compact Jacobi operator,
one may ask under which conditions $0$ is even an eigenvalue (necessarily
simple). An answer can be deduced directly from the eigenvalue equation
(\ref{eq:eigenvalue}). One immediately finds that any eigenvector
$x$ must satisfy $x_{2k}=0$ and $x_{2k-1}=(-1)^{k+1}x_{1}/\gamma_{2k-1}$,
$k\in\mathbb{N}$. Consequently, zero is a simple eigenvalue of $J$
iff \begin{equation}
\sum_{k=1}^{\infty}\frac{1}{\gamma_{2k-1}^{\,\,2}}<\infty.\label{eq:0_simple}\end{equation}

\begin{example} \label{example:4_1_over_nsq_diag0} Let $\lambda_{n}=0$
and $w_{n}=1/\sqrt{(n+\alpha)(n+\alpha+1)}$, $n\in\mathbb{N}$, where
$\alpha>-1$ is fixed. According to (\ref{eq:BesselJ_rel_F}) one
has, for $k\in\mathbb{Z}_{+}$, \[
\mathfrak{F}\!\left(\left\{ \frac{\gamma_{n}^{\,2}}{z}\right\} _{n=k+1}^{\infty}\right)=\mathfrak{F}\!\left(\left\{ \frac{1}{z(\alpha+n)}\right\} _{n=k+1}^{\infty}\right)=\Gamma(\alpha+k+1)z^{\alpha+k}J_{\alpha+k}\!\left(\frac{2}{z}\right)\!.\]
Hence\[
\spec(J)=\left\{ z\in\mathbb{R\setminus}\{0\};\, J_{\alpha}\!\left(\frac{2}{z}\right)=0\right\} \cup\{0\}.\]
The $k$th entry of an eigenvector $v(z)$ corresponding to a nonzero
eigenvalue $z$ may be written in the form\[
v_{k}(z)=\sqrt{\alpha+k}\, z^{\alpha}J_{\alpha+k}\!\left(\frac{2}{z}\right)\!,\ k\in\mathbb{N}.\]

It is well known that for $\alpha\in1/2+\mathbb{Z}$, the Bessel function
$J_{\alpha}(z)$ can be expressed as a linear combination of sine
and cosine functions, the simplest cases being \[
J_{-1/2}(z)=\sqrt{\frac{2}{\pi z}}\,\cos(z),\ J_{1/2}(z)=\sqrt{\frac{2}{\pi z}}\,\sin(z).\]
Thus for $\alpha=\pm1/2$ the spectrum of $J$ is described fully
explicitly. In other cases the eigenvalues of $J$ close to zero can
approximately be determined from the known asymptotic formulas for
large zeros of Bessel functions, see \cite[Eq.~9.5.12]{AbramowitzStegun}.
\end{example}

\begin{example} \label{example:5_power_1_diag0} Suppose that $0<q<1$
and put $\lambda_{n}=0$, $w_{n}=q^{n-1}$, $n\in\mathbb{N}$. With
the aid of (\ref{eq:F_geom_tw}) one derives that\[
\mathfrak{F}\!\left(\left\{ \frac{\gamma_{n}^{\,2}}{z}\right\} _{n=k+1}^{\infty}\right)={}_{0}\phi_{1}(;0;q^{2},-q^{2k}z^{-2}),\ k\in\mathbb{Z}_{+}.\]
It follows that\[
\spec(J)=\{z\in\mathbb{R}\setminus\{0\};\,{}_{0}\phi_{1}(;0;q^{2},-z^{-2})=0\}\cup\{0\}.\]
The components of an eigenvector $v(z)$ corresponding to an eigenvalue
$z\neq0$ may be expressed as\[
v_{k}(z)=q^{(k-1)(k-2)/2}z^{-k+1}\,{}_{0}\phi_{1}(;0;q^{2},-q^{2k}z^{-2}),\ k\in\mathbb{N}.\]

In this example as well as in the previous one, $0$ belongs to the
continuous spectrum of $J$ since the condition (\ref{eq:0_simple})
is not fulfilled. \end{example}

\subsection{Applications of Proposition~\ref{prop:norm_xi}}

Here we apply identity (\ref{eq:xi_sum}) to the five examples of
Jacobi matrices described above. Without going into details, the final
form of the presented identities is achieved after some simple substitutions.
On the other hand, no attempt is made here to optimize the range of
involved parameters; it is basically the same as it was for the Jacobi
matrix in question.

\smallskip\noindent1) In case of Example~\ref{example:1_lindag}
one gets \[
\sum_{k=1}^{\infty}J_{\nu+k}(x)^{2}=\frac{x}{2}\left(J_{\nu}(x)\,\frac{\partial}{\partial\nu}J_{\nu+1}(x)-J_{\nu+1}(x)\,\frac{\partial}{\partial\nu}J_{\nu}(x)\right)\!,\]
where $x>0$ and $\nu\in\mathbb{C}$. This is in fact a particular
case of (\ref{eq:sum_Bessel_mn}).

\smallskip\noindent2) In case of Example~\ref{example:2_1over_n}
one gets \[
\sum_{k=1}^{\infty}k\, J_{-\alpha z+k}(z)^{2}=\frac{z^{2}}{2}\left(J_{-\alpha z}(z)\,\frac{\mbox{d}}{\mbox{d}z}J_{-\alpha z+1}(z)-J_{-\alpha z+1}(z)\,\frac{\mbox{d}}{\mbox{d}z}J_{-\alpha z}(z)\right)\!,\]
where $\alpha>0$ and $z\in\mathbb{C}$.

\smallskip\noindent3) In case of Example~\ref{example:4_1_over_nsq_diag0}
one gets \[
\sum_{k=1}^{\infty}\,(\alpha+k)J_{\alpha+k}(z)^{2}=\frac{z^{2}}{2}\left(J_{\alpha}(z)\,\frac{\mbox{d}}{\mbox{d}z}J_{\alpha+1}(z)-J_{\alpha+1}(z)\,\frac{\mbox{d}}{\mbox{d}z}J_{\alpha}(z)\right)\!,\]
where $\alpha>-1$ and $z\in\mathbb{C}$.

\smallskip\noindent4) In case of Example~\ref{example:3_power_q}
one gets\begin{eqnarray*}
 &  & \hskip-3em\sum_{k=1}^{\infty}q^{(k-1)(k-2)/2}\,(\beta z)^{2k-2}\left((q^{k}z;q)_{\infty}\,\,{}_{0}\phi_{1}(;q^{k}z;q,-q^{k}\beta^{2}z^{2})\right)^{2}\\
 &  & \hskip-3em=(qz;q)_{\infty}^{\,\,\,2}\,\bigg(\,{}_{0}\phi_{1}(;z;q,-\beta^{2}z^{2})\,{}_{0}\phi_{1}(;qz;q,-q\beta^{2}z^{2})\\
 &  & \qquad\ +\, z\,(z-1)\Big({}_{0}\phi_{1}(;qz;q,-q\beta^{2}z^{2})\,\frac{\text{d}}{\mbox{d}z}\,{}_{0}\phi_{1}(;z;q,-\beta^{2}z^{2})\\
 &  & \qquad\qquad\qquad\qquad-\,{}_{0}\phi_{1}(;z;q,-\beta^{2}z^{2})\,\frac{\text{d}}{\mbox{d}z}\,{}_{0}\phi_{1}(;qz;q,-q\beta^{2}z^{2})\Big)\bigg),\end{eqnarray*}
where $0<q<1$, $\beta>0$ and $z\in\mathbb{C}$.

\smallskip\noindent5) In case of Example~\ref{example:5_power_1_diag0}
 one gets\begin{eqnarray*}
 &  & \sum_{k=1}^{\infty}q^{(k-1)(k-2)/2}\, z^{k-1}\,{}_{0}\phi_{1}(;0;q,-q^{k}z)^{2}\,=\,\,{}_{0}\phi_{1}(;0;q,-z)\,{}_{0}\phi_{1}(;0;q,-qz)\\
 &  & \qquad+\,2z\bigg({}_{0}\phi_{1}(;0;q,-z)\,\frac{\mbox{d}}{\mbox{d}z}\,{}_{0}\phi_{1}(;0;q,-qz)-\,{}_{0}\phi_{1}(;0;q,-qz)\,\frac{\mbox{d}}{\mbox{d}z}\,{}_{0}\phi_{1}(;0;q,-z)\!\bigg),\end{eqnarray*}
where $0<q<1$ and $z\in\mathbb{C}$.

\subsection{A Jacobi matrix with a linear diagonal and constant parallels \label{subsec:example_lin_diag}}

Here we discuss in somewhat more detail Example~\ref{example:1_lindag}
concerned with a Jacobi matrix having a linear diagonal and constant
parallels. For simplicity and with no loss of generality we put $\alpha=1$.
Our goal is to study how the spectrum of the Jacobi operator $J$
depends on the real parameter $w$. We treat $J$ as a linear operator-valued
function, $J=J(w)$. One may write $J(w)=L+wT$ where $L$ is the
diagonal operator with the diagonal sequence $\lambda_{n}=n$, $\forall n\in\mathbb{N}$,
and $T$ has all units on the parallels neighboring to the diagonal
and all zeros elsewhere. Notice that $\|T\|\leq2$.

We know that $J(w)$ has, for all $w\in\mathbb{R}$, a semibounded
simple discrete spectrum. Let us enumerate the eigenvalues in ascending
order as $\lambda_{s}(w)$, $s\in\mathbb{N}$. From the standard perturbation
theory one infers that all functions $\lambda_{s}(w)$ are real analytic,
with $\lambda_{s}(0)=s$. Moreover, the functions $\lambda_{s}(w)$
are also known to be even and so we restrict $w$ to the positive
real half-axis. In Example~\ref{example:1_lindag} we learned that
for every $w>0$ fixed, the roots of the equation $J_{-z}(2w)=0$
are exactly $\lambda_{s}(w)$, $s\in\mathbb{N}$. Several first eigenvalues
$\lambda_{s}(w)$ as functions of $w$ are depicted in Figure~1.

The problem of roots of a Bessel function depending on the order,
with the argument being fixed, has a long history. Here we make use
of some results derived in the classical paper \cite{Coulomb}. Some
numerical aspects of the problem are discussed in \cite{Ikebe_atal}.
For comparatively recent results in this domain one may consult \cite{Petropoulou_etal}
and references therein.

In \cite{Coulomb} it is shown that\[
\frac{d\lambda_{s}(w)}{dw}=-\left(2w\int_{0}^{\infty}K_{0}(4w\sinh(t))\exp\left(2\lambda_{s}(w)t\right)dt\right)^{\!-1}.\]
From this relation one immediately deduces a few basic qualitative
properties of the spectrum of the Jacobi operator.

\begin{proposition}[M.~J.~Coulomb] The spectrum $\{\lambda_{s}(w);\: s\in\mathbb{N}\}$
of the above introduced Jacobi operator, depending on the parameter
$w\geq0$, has the following properties.

(i) For every $s\in\mathbb{N}$, the function $\lambda_{s}(w)$ is
strictly decreasing.

(ii) If $r<s$ then $\lambda_{r}{}'(w)<\lambda_{s}{}'(w)$.

(iii) In particular, the distance between two neighboring eigenvalues
$\lambda_{s+1}(w)-\lambda_{s}(w)$, $s\in\mathbb{N}$, increases with
increasing $w$ and is always greater than or equal to $1$, with
the equality only for $w=0$. \end{proposition}

Let us next check the asymptotic behavior of $\lambda_{s}(w)$ at
infinity. The asymptotic expansion at infinity of the $s$th root
$j_{s}(\nu)$ of the equation $J_{\nu}(x)=0$, with $\nu$ being fixed,
reads \cite[Eq.~9.5.22]{AbramowitzStegun} \[
j_{s}(\nu)=\nu-2^{-1/3}a_{s}\nu^{1/3}+O\!\left(\nu^{-1/3}\right)\ \text{as}\ \nu\to+\infty,\]
where $a_{s}$ is the $s$th negative zero of the Airy function $\text{Ai}(x)$.
From here one deduces that \[
\lambda_{s}(w)=-2w-a_{s}w^{1/3}+O\left(w^{-1/3}\right)\text{ }\text{as}\text{ }w\to+\infty.\]

Concerning the asymptotic behavior of $\lambda_{s}(w)$ at $w=0$,
one may use the expression for the Bessel function as a power series
and apply the substitution, $\lambda_{s}(w)=s-z(w)$, $s=1,2,3,\ldots$.
The solution $z=z(w)$, with $z(0)=0$, is then defined implicitly
near $w=0$ by the equation\[
\sum_{m=0}^{\infty}\frac{(-1)^{m}}{m!\,\Gamma(m+1-s+z(w))}\, w^{2m}=0.\]
The computation is straightforward and based on the relation\[
\frac{1}{\Gamma(-m+z)}=(-1)^{m}m!\left(z-\psi^{(0)}(m+1)z^{2}\right)+O\!\left(z^{3}\right),\ m=0,1,2,3,\ldots,\]
where $\psi^{(0)}$ is the polygamma function. This way one derives
that, as $w\to0$,\begin{eqnarray}
\lambda_{1}(w) & = & 1-w^{2}+\frac{1}{2}w^{4}+O\!\left(w^{6}\right),\label{eq:asympt_lbd_s}\\
\lambda_{s}(w) & = & s-\frac{1}{(s-1)!s!}\, w^{2s}+\frac{2s}{(s-1)(s-1)!(s+1)!}\, w^{2s+2}+O\!\left(w^{2s+4}\right)\!,\text{ }\ \text{for}\ s\geq2.\nonumber \end{eqnarray}

The same asymptotic formulas, as given in (\ref{eq:asympt_lbd_s}),
can also be derived using the standard perturbation theory \cite[\S~II.2]{Kato}.
Alternatively, one may use equivalent formulas for coefficients of
the perturbation series derived in \cite{DSV1,DSV2} which are perhaps
more convenient for this particular example.

The distance of $s\in\mathbb{N}$ to the rest of the spectrum of of
the diagonal operator $L$ equals $1$. The Kato-Rellich theorem tells
us that there exists exactly one eigenvalue of $J(w)$ in the disk
centered at $s$ and with radius $1/2$ as long as $|w|<1/4$. The
explicit expression for the leading term in (\ref{eq:asympt_lbd_s})
suggests, however, that the eigenvalue $\lambda_{s}(w)$ may stay
close to $s$ on a much larger interval at least for high orders $s$.
It turns out that actually $\lambda_{s}(w)$ is well approximated
by this leading asymptotic term on an interval $[0,\beta_{s})$, with
$\beta_{s}\sim s/e$ for $s\gg1$. A precise formulation is given
in Proposition~\ref{prop:lbd_s_asympt_0} below.

Denote by $y_{k}(\nu)$ the $k$th root of the Bessel function $Y_{\nu}(z)$,
$k\in\mathbb{N}$. Let us put\[
\beta_{s}:=\left(\frac{(s-1)!\, s!}{\pi}\right)^{\!1/(2s)}\!,\ \ s\in\mathbb{N}.\]
In order to avoid confusion with the usual notation for Bessel functions,
the $n$th truncation of $J(w)$ is now denoted by a bold letter as
$\pmb{J}_{n}(w)$.

\begin{lemma} \label{lem:beta_s_leq_y1} The following estimate holds
true:

\begin{equation}
\beta_{s}<\frac{1}{2}\, y_{1}\!\!\left(s-\frac{1}{2}\right)\!,\ \forall s\in\mathbb{N}.\label{eq:beta_s_leq_y1}\end{equation}
\end{lemma}

\begin{proof} One knows that $\nu<y_{1}(\nu)$, $\forall\nu\geq0$
\cite[Eq.~9.5.2]{AbramowitzStegun}, and in particular this is true
for $\nu=s-1/2$, $s\in\mathbb{N}$. On the other hand, the sequence\[
\phi_{s}=\frac{\pi}{(s-1)!\, s!}\left(s-\frac{1}{2}\right)^{\!2s}2^{-2s}\]
is readily verified to be increasing, and $1<\phi_{4}$. This shows
(\ref{eq:beta_s_leq_y1}) for all $s\geq4$. The cases $s=1,2,3$
may be checked numerically. \end{proof}

\begin{lemma} \label{lem:eq_lbd_s} Denote by $\chi_{n}(w;z)$ the
characteristic polynomial of the $n$th truncation $\pmb{J}_{n}(w)$
of the Jacobi matrix $J(w)$. If $0\leq w\leq\beta_{s}$ for some
$s\in\mathbb{N}$ then $z=\lambda_{s}(w)$ solves the equation\begin{equation}
\chi_{2s-1}(w;z)-\frac{wJ_{2s-z}(2w)}{J_{2s-1-z}(2w)}\,\chi_{2s-2}(w;z)=0.\label{eq:eigenval_chi_linear}\end{equation}
 \end{lemma}

\begin{proof} Let $\left\{ e_{k};\, k\in\mathbb{N}\right\} $ be
the standard basis in $\ell^{2}(\mathbb{N})$. Let us split the Hilbert
space into the orthogonal sum\[
\ell^{2}(\mathbb{N})=\text{span}\left\{ e_{k};\,1\leq k\leq2s-1\right\} \oplus\overline{\text{span}\left\{ e_{k};\,2s\leq k\right\} }.\]
Then $J(w)$ splits correspondingly into four matrix blocks,\[
J(w)=\left(\begin{array}{cc}
A(w) & B(w)\\
C(w) & D(w)\end{array}\right)\!.\]
Here $A(w)=\pmb{J}_{2s-1}(w)$, $D(w)=J(w)+(2s-1)I$, the block $B(w)$
has just one nonzero element in the lower left corner and $C(w)$
is transposed to $B(w)$.

By the min max principle, the minimal eigenvalue of $D(w)$ is greater
than or equal to $2s-2w$. Since $\lambda_{s}(w)\leq s$ one can estimate\[
\min\text{spec}(D(w))-\lambda_{s}(w)=\lambda_{1}(w)-\lambda_{s}(w)+2s-1\geq s-2w.\]
We claim that $0\leq w\leq\beta_{s}$ implies $\min\text{spec}(D(w))-\lambda_{s}(w)>0$.
This is obvious for $s=1$. For $s\geq2$, it suffices to show that
$\beta_{s}<s/2$. This can be readily done by induction in $s$. Hence,
under this assumption, $D(w)-z$ is invertible for $z=\lambda_{s}(w)$.

Solving the eigenvalue equation $J(w)\pmb{v}=z\pmb{v}$ one can write
the eigenvector as a sum $\pmb{v}=\pmb{x}+\pmb{y}$, in accordance
with the above orthogonal decomposition. If $D(w)-z$ is invertible
then the eigenvalue equation reduces to the finite-dimensional linear
system\begin{equation}
\left(A-z-B(D-z)^{-1}C\right)\pmb{x}=0.\label{eq:eigenval_red_linear}\end{equation}

One observes that $B(D-z)^{-1}C$ has all entries equal to zero except
of the element in the lower right corner. Using (\ref{eq:Weyl_mfce})
and (\ref{eq:BesselJ_rel_F}) one finds that this nonzero entry equals\[
wJ_{2s-z}(2w)J_{2s-1-z}(2w)^{-1}.\]
Equation (\ref{eq:eigenval_chi_linear}) then immediately follows
from (\ref{eq:eigenval_red_linear}). \end{proof}

\begin{proposition} \label{prop:lbd_s_asympt_0} For $s\in\mathbb{N}$
and $0\leq w\leq\beta_{s}$, one has

\[
0\leq s-\lambda_{s}(w)\leq\frac{1}{\pi}\,\arcsin\!\left(\frac{\pi w^{2s}}{(s-1)!s!}\right)\!.\]
\end{proposition}

\begin{proof} We start from Lemma~\ref{lem:eq_lbd_s} and equation
(\ref{eq:eigenval_chi_linear}). Let us recall from \cite[Proposition~30]{StampachStovicek}
that\[
\det\left(\pmb{J}_{2s-1}(w)-s-x\right)=(-1)^{s}x\,\sum_{k=0}^{s-1}\binom{2s-k-1}{k}w^{2k}\prod_{j=1}^{s-k-1}\left(j^{2}-x^{2}\right)\!.\]
Hence if $z\in\mathbb{R}$, $|z-s|\leq1$, then\begin{equation}
\left|\chi_{2s-1}(w;z)\right|\geq|z-s|\,\prod_{j=1}^{s-1}\left(j^{2}-(z-s)^{2}\right)\!.\label{eq:estim_chi_2s-1}\end{equation}

Since $J_{-s+1/2}(x)=(-1)^{s}\, Y_{s-1/2}(x)$ it is true that for
$2w=y_{1}(s-1/2)$ one has $\lambda_{s}(w)=s-1/2$. Because of monotonicity
of $\lambda_{s}(w)$ one makes the following observation: if $2w\leq y_{1}(s-1/2)$
then $s\geq\lambda_{s}(w)\geq s-1/2$.

By Lemma~\ref{lem:beta_s_leq_y1}, if $w\leq\beta_{s}$ then $2w\leq y_{1}(s-1/2)$,
and so the estimate (\ref{eq:estim_chi_2s-1}) applies for $z=\lambda_{s}(w)$.
Using also Proposition~\ref{prop:charfce_finJ} to express $\chi_{2s-2}(w;z)$
one derives from (\ref{eq:eigenval_chi_linear}) that\begin{equation}
|\lambda-s|\leq w\left|\frac{J_{2s-\lambda}(2w)}{J_{2s-1-\lambda}(2w)}\right|\left|\frac{s-\lambda}{2s-1-\lambda}\,\mathfrak{F}\!\left(\frac{w}{1-\lambda},\frac{w}{2-\lambda},\ldots,\frac{w}{2s-2-\lambda}\right)\right|\label{eq:1st_estim_lbd_s}\end{equation}
where as well as in the remainder of the proof we write for short
$\lambda$ instead of $\lambda_{s}(w)$.

Starting from the equation

\[
\mathfrak{F}\!\left(\frac{w}{1-\lambda},\frac{w}{2-\lambda},\frac{w}{3-\lambda},\ldots\right)=0,\ \ \text{with}\ \lambda=\lambda_{s}(w),\]
and using (\ref{eq:F_T_recur_k}), (\ref{eq:BesselJ_rel_F}) one derives
that, for all $k\in\mathbb{Z}_{+}$,\begin{equation}
\left(\prod_{j=1}^{k}(j-\lambda)\right)\mathfrak{F}\!\left(\frac{w}{1-\lambda},\frac{w}{2-\lambda},\ldots,\frac{w}{k-\lambda}\right)=w^{k}\,\frac{J_{k+1-\lambda}(2w)}{J_{1-\lambda}(2w)}\,.\label{eq:F_eq_ration_Bessel}\end{equation}
Combining (\ref{eq:1st_estim_lbd_s}) and (\ref{eq:F_eq_ration_Bessel})
we get (knowing that $0\leq s-\lambda\leq1/2$ for $\lambda=\lambda_{s}(w)$)\[
s-\lambda\leq w^{2s-1}\left|\left(\prod_{j=1}^{s-1}(\lambda-j)\text{ }\prod_{j=1}^{s-1}(j+s-\lambda)\right)^{\!\!-1}\frac{J_{2s-\lambda}(2w)}{J_{1-\lambda}(2w)}\right|.\]
But notice that, by expressing the sine function as an infinite product,\[
\prod_{j=1}^{s-1}(\lambda-j)\text{ }\prod_{j=1}^{s-1}(j+s-\lambda)=((s-1)!)^{2}\,\frac{\sin(\pi(s-\lambda))}{\pi(s-\lambda)}\left(\prod_{j=s}^{\infty}\left(1-\frac{(s-\lambda)^{2}}{j^{2}}\right)\right)^{\!-1}\!.\]
Hence\[
\sin(\pi(s-\lambda))\leq\pi\,\frac{w^{2s-1}}{((s-1)!)^{2}}\left|\frac{J_{2s-\lambda}(2w)}{J_{1-\lambda}(2w)}\right|\!.\]

From (\ref{eq:Jnu_Jnu_eq_sin}) one gets, while taking into account
that $J_{-\lambda}(2w)=0$,\[
\sin(\pi\lambda)=\pi wJ_{\lambda}(2w)J_{1-\lambda}(2w).\]
In addition, one knows that\[
\left|J_{\nu}(x)\right|\leq\frac{1}{\Gamma(\nu+1)}\left|\frac{x}{2}\right|^{\nu}\]
provided $\nu>-1/2$ and $x\in\mathbb{R}$ \cite[Eq.~9.1.62]{AbramowitzStegun}.
Hence\[
\sin(\pi(s-\lambda))^{2}\leq\frac{\pi^{2}w^{4s}}{((s-1)!)^{2}\,\Gamma(2s+1-\lambda)\Gamma(\lambda+1)}\,.\]
Writing $\lambda=s-\zeta$, with $0\leq\zeta\leq1/2$, one has\[
\frac{d}{d\zeta}\log\!\left(\frac{1}{\Gamma(s+\zeta+1)\Gamma(s-\zeta+1)}\right)=-\psi^{(0)}(s+\zeta+1)+\psi^{(0)}(s-\zeta+1)<0.\]
Thus we arrive at the estimate\[
\sin(\pi(s-\lambda))^{2}\leq\frac{\pi^{2}w^{4s}}{((s-1)!)^{2}\,(s!)^{2}}\,.\]

To complete the proof it suffices to notice that the assumption $w\leq\beta_{s}$
means nothing but $\left.w^{2s}\right/((s-1)!s!)\leq1$, and it also
implies that $0\leq s-\lambda\leq1/2$. \end{proof}

\section*{Acknowledgments}

The authors wish to acknowledge gratefully partial support from the
following grants: Grant No.\ 201/09/0811 of the Czech Science Foundation
(P.\v{S}.) and Grant No.\ LC06002 of the Ministry of Education of
the Czech Republic (F.\v{S}.).

\clearpage{}\includegraphics[bb=0bp 0bp 450bp 261bp]{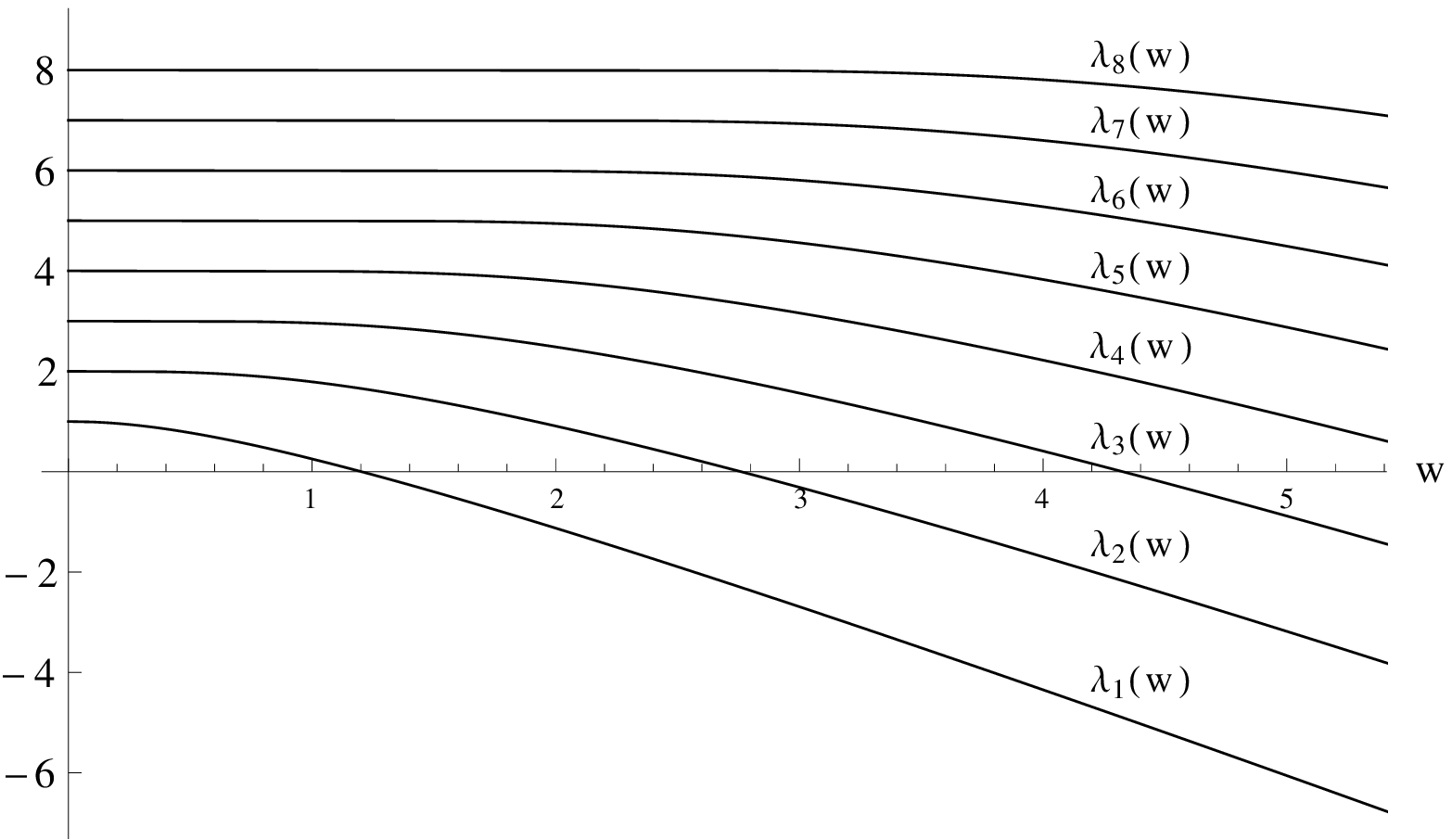}

\vspace{3\baselineskip}

\noindent Figure~1. Several first eigenvalues $\lambda_{s}(w)$
as functions of the parameter $w$ for the Jacobi operator $J=J(w)$
from Example~\ref{example:1_lindag}, with $\alpha=1$.
\end{document}